\documentclass[11pt]{amsart}
\usepackage{graphicx}
\usepackage{amsmath, amsthm, amssymb}

\theoremstyle{definition}
\newtheorem{thm}{Theorem}[section]
\newtheorem{lem}[thm]{Lemma}
\newtheorem{prop}[thm]{Proposition}

\newtheorem{ex}[thm]{Example}

\newtheorem*{defin}{Definition}
\newtheorem*{rem}{Remark}

\begin{document}

\title[Dual graphs and generating sequences]{Dual graphs and generating sequences of non-divisorial valuations on two-dimensional function fields}

\author{Charles Li}

\address{Department of Mathematics and Computer Sciences, Mercy College, 555 Broadway, Dobbs Ferry, NY 10522, USA}

\email{cli2@mercy.edu}


\date{}

\maketitle

\begin{abstract}
An exposition on Spivakovsky's dual graphs of valuations on function fields of dimension two is first given, leading to an alternative proof of minimal generating sequences for the non-divisorial valuations. It should be noted that the definition of generating sequence used in this paper is different from Spivakovsky's original usage. This change leads to an explicit formulation of generating sequence values for the non-divisorial cases in terms of data from their dual graphs. The proofs are elementary in the sense that only continued fractions and the linear Diophantine Frobenius problem from classical number theory are used.
\end{abstract}

\section{Introduction}

There are two main goals in this paper: 1) provide an exposition on Spivakovsky's dual graphs of valuations on function fields of surfaces with a focus on the non-divisorial valuations in particular, and 2) use the dual graphs to give a proof of minimal generating sequences for non-divisorial valuations.

It should be noted that the definition of generating sequence used in this paper is different from Spivakovsky's original usage. This change leads to an explicit formulation of generating sequence values for the non-divisorial cases in terms of data from their dual graphs. In this sense, this paper offers a variation on Spivakovsky's work.

One motivation for the study of dual graphs and generating sequences of valuations is the problem of resolution of singularities. Zariski used valuation theory to solve this classical problem in dimensions two and three in characteristic 0 by studying sets of valuations on a field $K$, centered on a subring $R$, i.e. a Zariski-Riemann space. The dual graph of a valuation is a nice way of visualizing a valuation and lends itself to the study of sets of valuations through sets of dual graphs. This approach is equivalent to the valuative tree in \cite{fj}.

Another way of studying valuations is through their Poincar\'e series. This paper sets up another paper (already published as \cite{lisch}) that explores the connections between the dual graphs, generating sequences and Poincar\'e series of non-divisorial valuations on function fields of dimension two.

The following setting will be used throughout. Let $(R,\mathfrak{m})$ be a two-dimensional regular local ring whose fraction field $K$ is a function field of dimension two (i.e. of transcendence degree two) over an algebraically closed base field $k$ of characteristic 0. Let $\nu$ be a valuation on $K$ that's centered on $R$, i.e. $\mathfrak{m}=\mathfrak{m}_{V}\cap R$, where $\mathfrak{m}_V$ is the maximal ideal of the valuation ring $V$, and $\text{Frac}(R)=\text{Frac}(V)=K$. Hence, $\nu$ is a map from $K$ to an ordered abelian group $\Gamma \cup \{\infty\}$  satisfying the axioms:
\begin{equation*}
\begin{array}{ll}
1) & \nu(xy)=\nu(x)+\nu(y) \\
2) & \nu(x+y)\geq \min\{\nu(x),\nu(y)\}\\
3) & \nu(x)=\infty \iff x=0
\end{array}
\end{equation*}
for all $x,y\in K$. As usual, $\gamma+\infty=\infty+\gamma=\infty$ for all $\gamma\in\Gamma$. Lastly, $\nu(c)=0$ for all $c\in k\setminus{0}$.

It is useful to first see an overview of this paper before diving into the details. The reader may wish to read this paragraph lightly at first and return to it later if it helps in better grasping the bigger picture. Briefly speaking, the valuations in the setting described above can be interpreted as encoding information about sequences of point blowups. Algebraically, a valuation $\nu$ determines a sequence of regular local ring extensions:
$$R_0\rightarrow R_1\rightarrow R_2 \rightarrow \cdots \rightarrow V$$
where $R_0:=R$. Let $\mathfrak{m}_i=(x_i,y_i)$ denote the maximal ideal of $R_i$, where the parameters $x_i$ and $y_i$ are obtained from the previous level via special rules determined by $\nu$. The intersections of the exceptional components of the exceptional set give rise to the dual graph of the underlying valuation when the graph theoretic dual of the reduced exceptional set is considered. These exceptional components are given in local equations by the regular parameters of $\{R_i\}$. Carefully tracing the values of the parameters along the sequence of blowups allows us to see how the exceptional components intersect, which results in the shape of the dual graph, and also results in the values of the elements of a generating sequence associated to the valuation $\nu$. An analysis of the parameter values will show that generating sequences of a certain form for non-divisorial valuations are minimal. Continued fractions and the Frobenius problem will be used in the details.

\section{dual graphs} \label{sectdg}

Dual graphs of valuations are a nice combinatorial way of visualizing valuations. They were introduced by Spivakovsky in \cite{spiv}. In this section, we will review blowups and dual graphs as well as establish notation. We differ from Spivakovsky's exposition by stressing the local perspective.

Recall, we have a point associated with the maximal ideal $\mathfrak{m}$. This point is blown up in a sequence of point blowups along the valuation $\nu$. Algebraically, there is a sequence of local ring extensions:
\begin{equation*}
\pi^{*}: R=R_0 \xrightarrow{\pi_1^*} R_1 \xrightarrow{\pi_2^*} \cdots \rightarrow R_i \xrightarrow{\pi_{i+1}^*} R_{i+1} \rightarrow \cdots \rightarrow V
\end{equation*}
where we get $V$ in the limit by a consequence of local uniformization in dimension two (Theorem \ref{lu}).
The points in this sequence are called {\it centers}. The center corresponding to the maximal ideal $\mathfrak{m}_i$ is denoted $\eta_i$. The parameters of the maximal ideal $(x_i,y_i)=\mathfrak{m}_i$ determine $(x_{i+1},y_{i+1})=\mathfrak{m}_{i+1}$ in one of three ways depending on $\nu(x_i)$ and $\nu(y_i)$:
\begin{equation*}
\begin{array}{llll}
1) & x_{i} = x_{i+1}, & y_{i} = x_{i+1}y_{i+1} & (x\text{-blowup})
\\\\
2) & x_{i} = x_{i+1}y_{i+1}, & y_{i}=y_{i+1} & (y\text{-blowup})
\\\\
3) & x_{i} = x_{i+1}, & y_{i} = x_{i+1}\left(y_{i+1}+c_i\right) & (z\text{-blowup})
\end{array}
\end{equation*}
where $c_i\neq 0$ is the residue of $y_{i}/x_{i}$ in $k$.

If $\nu(x_i)<\nu(y_i)$, then an $x$-blowup is used. If $\nu(x_i)>\nu(y_i)$, then a $y$-blowup is used. If $\nu(x_i)=\nu(y_i)$, then two cases arise depending on whether the next point blowup is the last in the sequence of blowups. If the sequence of blowups ends at level $i+1$, then $R_{i+1}=V$ is a discrete valuation ring and an $x$-blowup is used to determine the transformation of the parameters from level $i$ to level $i+1$. Furthermore $\mathfrak{m}_{i+1}=(x_{i+1})$ since $y_{i+1}$ will be a unit. Otherwise, if $\nu(x_i)=\nu(y_i)$ and the sequence of blowups doesn't terminate at level $i+1$, then $R_{i+1}$ is not a discrete valuation ring and a $z$-blowup is used to transform the parameters.

These $x$-blowups, $y$-blowups and $z$-blowups introduce new elements to $R_i$ to get a new intermediate ring, say, $R_{i+1}^{'}$. These blowups introduce $y_i/x_i$, $x_i/y_i$ and $y_i/x_i-c_i$, respectively. Now to obtain $R_{i+1}$, localize $R_{i+1}^{'}$ at $\mathfrak{m}_{i+1}=\mathfrak{m}_V \cap R_{i+1}^{'}$. Taking an $x$-blowup as an example, $R_{i+1}^{'} = R_i[y_i/x_i]$ and $R_{i+1}=\left(R_i[y_i/x_i]\right)_{\mathfrak{m}_{i+1}}$. Notice $\nu$ is centered on every $R_i$ and that $\text{Frac}(R_i)=K$ for $i\geq 0$.

Geometrically, we have a sequence of maps:
\begin{equation*}
\pi: \text{Spec} V \rightarrow \cdots \rightarrow \text{Spec} R_{i+1} \xrightarrow{\pi_{i+1}} \text{Spec} R_i \rightarrow \cdots \xrightarrow{\pi_2} \text{Spec} R_1 \xrightarrow{\pi_1} \text{Spec} R_0
\end{equation*}

The {\it exceptional set} is defined to be $\pi^{-1}(\eta_0)$, where the 0-th center $\eta_0$ is the point corresponding to $\mathfrak{m}$. An {\it exceptional component} $L_i$ of the exceptional set is defined to be $L_i=\pi_i^{-1}(\eta_{i-1})\in\text{Spec}R_i$. Points on $L_i$ are considered {\it infinitely near} the previous $\eta_j$, for $j<i$.

The intuition and terminology behind point blowups comes from classical algebraic geometry. The idea is to replace a point with a line that represents tangent directions at the point, and this line is considered to be the infinitely near neighborhood to the point. Blowing up the origin in the $xy$-plane with an $x$-blowup introduces $y_1=y/x$, which encodes information about slopes of tangent lines through the origin $\eta_0$. If we work with $k=\mathbb{C}$ and projectivize, then the projective line $\mathbb{P}^{1}(\mathbb{C})$ can be visualized as a sphere, hence the name ``blowup'' in the sense of inserting a straw and blowing up a bubble at the point.

It is useful to think about what happens in the real setting $\mathbb{R}^2$. When the origin of the $xy$-plane is blown up, the $xy$-plane is pulled and twisted in the third dimension. In the process, the exceptional component $L_1$ was formed, blown up from the origin. We can ``flatten'' this picture out and think of it as the $x_1y_1$-plane for intuition, where $L_1$ is the $y_1$-axis and has local equation $x_1=0$ (assuming an $x$-blowup was used), and where the $x_1$-axis is given by $y_1=0$. The points on the $y_1$-axis correspond to slopes of lines through the origin in the $xy$-plane. The $y$-axis in the $xy$-plane has slope $\infty$ and hence is represented by a point at infinity with respect to the $y_1$-axis. A second point blowup can be performed at a point on the $y_1$-axis, yielding a new $x_2y_2$-plane. Continuing in this manner, an infinite sequence of point blowups can be visualized.

If an $x$-blowup is used to go from level $i$ to level $i+1$, then the new exceptional component $L_{i+1}$ will be given in local equations by $x_{i+1}=0$, i.e. $L_{i+1}:x_{i+1}=0$. If a $y$-blowup is used, then we have $L_{i+1} : y_{i+1}=0$. If a $z$-blowup is used, then we also have $L_{i+1}:x_{i+1}=0$. Geometrically, the $z$-blowup is different from the $x$-blowup (or the $y$-blowup) in that the $z$-blowup sets up a point blowup occurring at a center $\eta_{i+1}$ on $L_{i+1}$ that's different from the ``origin'' of the $y_i/x_i$-axis (or $x_i/y_i$-axis, respectively).

The effect of blowups on curves is important and so we now set notation. Let $C$ be a curve in $\text{Spec}R_i$, hence $\eta_i \in C$. Let $j>i$. The {\it total transform} of $C$ after the $j$-th blowup is $(\pi_{i+1} \circ \cdots \circ \pi_{j})^{-1}(C)$ . The {\it strict transform} of $C$ after the $j$-th blowup, denoted $C^{(j)}$, is the Zariski closure of $(\pi_{i+1} \circ \cdots \circ \pi_{j})^{-1}(C\setminus \eta_i)$. The {\it exceptional transform} of $C$ after the $j$-th blowup is $(\pi_{i+1} \circ \cdots \circ \pi_{j})^{-1}(\eta_i)$, where the exceptional components are ``counted properly.''

From the algebraic perspective, if $C$ is given by $f=0$, then the total transform after the $j$-th blowup is: $\displaystyle (\pi_j^*\circ\cdots\circ\pi_{i+1}^*)(f)=x_j^{e_1}y_j^{e_2}f^{(j)}$, where $e_1,e_2\in\mathbb{N}_0$, and $f^{(j)}$ is not divisible by $x_j$ or by $y_j$. The total transform is made of the strict transform $f^{(j)}$ and the exceptional transform $x_j^{e_1}y_j^{e_2}$. For simplicity, we will just write: $f=x_j^{e_1}y_j^{e_2}f^{(j)}$. It is easy to see that $\nu(f^{(n)})<\nu(f^{(m)})$ for $m<n$ if $\nu(f^{(m)})>0$. Similarly, it is also easy to see that $\nu(f^{(n)})=\nu(f^{(m)})=0$ for $m<n$ if $\nu(f^{(m)})=0$. The values of the strict transforms are monotonically decreasing as successive blowups are performed since one of the parameters, say $x_i$ or $y_i$, is being factored out into the exceptional transform at each stage, until the strict transform becomes a unit.

We are now ready to tackle the dual graphs of valuations. The dual graph of a valuation is a beautiful combinatorial object that represents a valuation via intersections of exceptional components of the exceptional divisor. The valuation is thus described through its effects on a point that is birationally transformed by blowups. The concept has its origins in Zariski's Main Theorem: the exceptional set is connected. The dual graph is the graph theoretic dual of the reduced exceptional set, inverting the exceptional components (lines) and intersections (points), to get vertices (exceptional components) and edges (intersection points), respectively. Dual graphs of valuations will be simple connected graphs.

Dual graphs are easiest to understand through examples. A concrete simple example of the dual graph of a divisorial valuation will hopefully add some clarity to the general process.

\begin{ex} \label{dgex} Consider a divisorial valuation $\nu$ such that $\nu(x)=1$, $\nu(y)=7/2$ and $\nu(y^2-x^{7})=43/6$. Note that $\displaystyle \frac{43}{6} = 7 + \frac{1}{6}$. Following the rules specified above on when to use $x$-blowups, $y$-blowups and $z$-blowups, we have the following data for a sequence of transformations:

\begin{equation*}
\begin{array}{|c|c|c|c|}
\hline
\nu(x_i) & \nu(y_i) & x \text{ transformation} & y \text{ transformation} \\ \hline
\nu(x_0)=1 & \nu(y_0)=7/2 & x_0=x_1 & y_0=x_1y_1 \\ \hline
\nu(x_1)=1 & \nu(y_1)=5/2 & x_1=x_2 & y_1=x_2y_2 \\ \hline
\nu(x_2)=1 & \nu(y_2)=3/2 & x_2=x_3 & y_2=x_3y_3 \\ \hline
\nu(x_3)=1 & \nu(y_3)=1/2 & x_3=x_4y_4 & y_3=y_4 \\ \hline
\nu(x_4)=1/2 & \nu(y_4)=1/2 & x_4=x_5 & y_4=x_5(y_5+1) \\ \hline
\nu(x_5)=1/2 & \nu(y_5)=1/6 & x_5=x_6y_6 & y_5=y_6 \\ \hline
\nu(x_6)=1/3 & \nu(y_6)=1/6 & x_6=x_7y_7 & y_6=y_7 \\ \hline
\nu(x_7)=1/6 & \nu(y_7)=1/6 & x_7=x_8 & y_7=x_8y_8 \\ \hline
\end{array}
\end{equation*}
\end{ex}

We stop after the 8th blowup and $R_8$ is a discrete valuation ring with uniformizing parameter $x_8$. Here $x_8=0$ gives the local equation for the exceptional component $L_8$, and $y_8$ is a unit. By convention, the last transformation was arbitrarily chosen to be an $x$-blowup instead of a $y$-blowup. In this example, $\nu$ counts the order of vanishing along $L_8$ multiplied with a normalization factor $1/b$, for some $b\in\mathbb{N}$, to account for normalizing the valuation such that $\nu(x)=1$. More precisely, for $f\in R$, we have $\displaystyle \nu(f)=\frac{1}{b}\text{ord}_{x_8}\left(\pi^{*}(f)\right)$, and in this example $b=6$. Notice the sequence of transformations is: 3 $x$-blowups, 1 $y$-blowup, 1 $z$-blowup, 2 $y$-blowups, and lastly 1 $x$-blowup.

The dual graph for this example will be built in stages. Exceptional components will be represented by vertices in the dual graph. The intersection between an exceptional component and a previous exceptional component -- or its strict transform -- will be represented by an edge connecting the two vertices corresponding to the exceptional components. Notice that once a strict transform of an exceptional component is a unit in some $R_i$, we no longer need to consider it for all future blowups since it will stay a unit. Geometrically, this corresponds to the strict transform being away from some center $\eta_i$, hence the strict transform will {\it not} be infinitely near the future centers $\eta_j$ in the sequence of blowups, where $j>i$.

After the first $x$-blowup, we get the exceptional component $L_1: x_1=0$. This is represented by a vertex, labeled vertex 1. After the second $x$-blowup, we get the exceptional component $L_2: x_2=0$. Now, $L_2$ intersects $L_1$ since they share the center $\eta_1$ in common. We now have two vertices joined by an edge in building the dual graph of $\nu$; vertex 1 is adjacent to vertex 2. See Figure \ref{ex1} which shows the first four steps in the process to build the dual graph.

\begin{figure}
\centering
\includegraphics[width=125mm, bb=0 0 628 78]{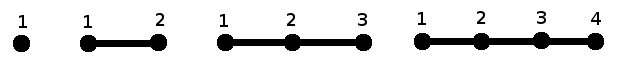}
\caption{The first four stages in building the dual graph}
\label{ex1}
\end{figure}

Notice $L_2:x_2=0$ and the $x$ transformation of the second $x$-blowup is $x_1=x_2*1$, so the strict transform $L_1^{(2)}$ is given by $1=0$, i.e. the strict transform is empty geometrically. By slight abuse of notation, we say $L_1^{(2)}$ is a unit when its local equation is given by a unit. We won't have to consider $L_1^{(i)}$ for $i\geq2$. Similarly, $L_3: x_3=0$ intersects $L_2$ at $\eta_2$, so vertex 3 is adjacent to vertex 2, but not adjacent to vertex 1 since $L_1^{(i)}$ is a unit in $R_i$ for $i\geq2$. The third $x$-blowup gives $x_2=x_3$ and the strict transforms $L_2^{(i)}$ are units so they can be ignored for $i\geq3$. The fourth blowup is a $y$-blowup which gives $L_4: y_4=0$. $L_4$ intersects $L_3$ at $\eta_3$. Vertex 4 is only adjacent to vertex 3 since $L_1^{(4)}$ and $L_2^{(4)}$ are units. The strict transform $L_3^{(4)}:x_4=0$ is not a unit. Both $L_4$ and $L_3^{(4)}$ are positively valued: $\nu(y_4)>0$ and $\nu(x_4)>0$. Thus, their intersection point is the center $\eta_4$ of the next blowup.

\begin{figure}
\centering
\includegraphics[width=95mm, bb=0 0 386 97]{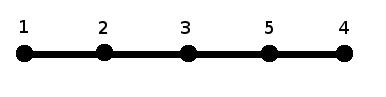}
\caption{After the fifth blowup}
\label{ex2}
\end{figure}

The exceptional component $L_5: x_5=0$ is represented by vertex 5 which is adjacent to both vertices 3 and 4. See Figure \ref{ex2}. Now $y_5=y_4/x_4-1$ and notice $\nu(y_4/x_4-1)>0$. Here we used 1 for the residue of $y_4/x_4$ in $R_5/\mathfrak{m}_5 \cong k$ for simplicity and this choice would not affect the resulting dual graph. Notice $L_3^{(5)}$ is a unit since $x_4=x_5$, and $L_4^{(5)}: y_5+1=0$ is also a unit. Thus, $L_3^{(i)}$ and $L_4^{(i)}$ will be ignored from now on. The center $\eta_5$ is determined by $(x_5,y_5)$. We have a $y$-blowup at $\eta_5$ since $\nu(y_5) < \nu(x_5)$. Now $L_6: y_6=0$ intersects $L_5$ at $\eta_5$ and so we have Figure \ref{ex3}.

\begin{figure}
\centering
\includegraphics[width=90mm, bb=0 0 385 125]{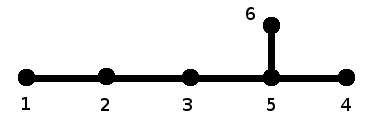}
\caption{After the sixth blowup}
\label{ex3}
\end{figure}

We wish to standardize the appearance of dual graphs, so let us adopt the convention that the graphs will open to the right and downward. As such, we shall rotate the rightmost portion of the dual graph when a node such as vertex 5 is introduced. We get Figure \ref{ex4}.

\begin{figure}
\centering
\includegraphics[width=90mm, bb=0 0 385 124]{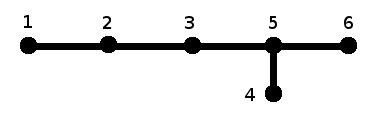}
\caption{After the sixth blowup, again}
\label{ex4}
\end{figure}

Continuing, $L_5^{(6)}: x_6=0$, which is not a unit. The center $\eta_6$ is the intersection of $L_6$ and $L_5^{(6)}$ so vertex 7 (corresponding to $L_7: y_7=0$) will be adjacent to both vertices 5 and 6. We get Figure \ref{ex5}.

\begin{figure}
\centering
\includegraphics[width=95mm, bb=0 0 434 109]{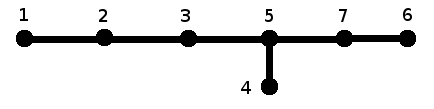}
\caption{After the seventh blowup}
\label{ex5}
\end{figure}

The strict transform $L_5^{(7)}:x_7=0$ is not a unit, while $L_6^{(7)}$ is a unit since $y_6=y_7$. Thus, vertex 8 will be adjacent to vertex 5. Vertex 8 will also be adjacent to vertex 7 since $L_8$ intersects $L_7$ at $\eta_7$. Notice that $\nu(x_7)=\nu(y_7)=1/6$ so this will be the last blowup before we reach the exceptional component $L_8 : x_8=0$ that determines the divisorial valuation. We distinguish the last vertex 8 by using an open dot. See Figure \ref{ex6}.

\begin{figure}
\centering
\includegraphics[width=110mm, bb=0 0 480 121]{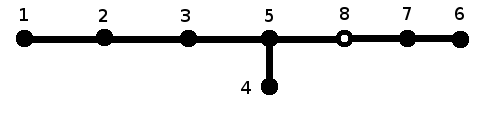}
\caption{After the eighth blowup}
\label{ex6}
\end{figure}

Now, to standardize the dual graph, we rotate the portion of the graph to the right of vertex 8 to get Figure \ref{ex7}, which is what we will call the dual graph of the valuation $\nu$ in the example.

\begin{figure}
\centering
\includegraphics[width=120mm, bb=0 0 630 347]{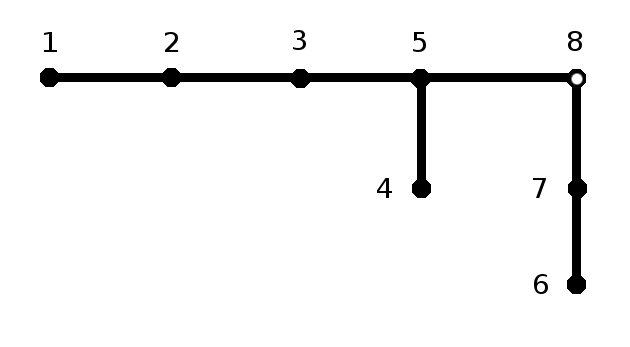}
\caption{The dual graph for the example}
\label{ex7}
\end{figure}

In general, if we rotate the graphs along the way as we did in the example so that the dual graph spreads to the right and downwards, then the dual graph naturally breaks up into ``L-shaped'' {\it dual graph pieces}, denoted $G_i$. Consider one such typical dual graph piece as in Figure \ref{Gi}. The figure shows $G_1$, the first piece of a dual graph $G=\bigcup G_i$.

\begin{figure}
\centering
\includegraphics[width=135mm, bb=0 0 630 380]{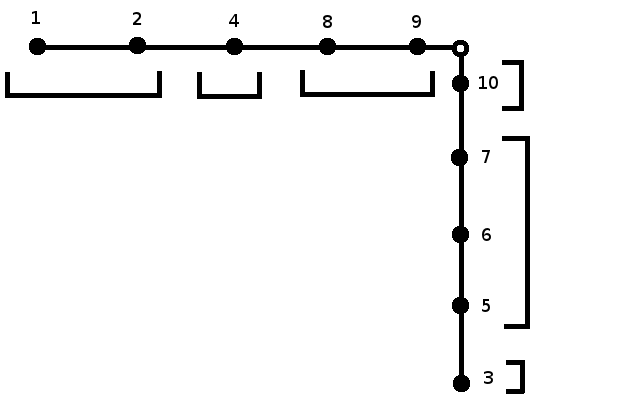}
\caption{Dual graph piece $G_1$}
\label{Gi}
\end{figure}

We call the horizontal portion the {\it odd leg} of $G_i$ and we call the vertical portion the {\it even leg}. There are $m_i$ segments of consecutively numbered vertices in $G_i$. In Figure \ref{Gi}, $m_1=6$. Let $a_j^{(i)}$ denote the number of vertices in the $j$-th segment of $G_i$. In Figure \ref{Gi}, $a_1^{(1)}=2$, $a_2^{(1)}=1$, $a_3^{(1)}=1$, $a_4^{(1)}=3$, $a_5^{(1)}=2$ and $a_6^{(1)}=1$. For computations, we exclude the last vertex denoted by the open dot. That particular vertex belongs to the next dual graph piece, $G_2$ in this case.

The very last dual graph piece could be of the form depicted in Figure \ref{Gi} or it could be of the form depicted in Figure {\ref{Ggplus1}, with only an odd leg. Let $g$ be the number of dual graph pieces with both odd and even legs (Figure \ref{Gi}). If the last dual graph piece has only an odd leg (Figure \ref{Ggplus1}), then the last piece is the $(g+1)$-th piece $G_{g+1}$. In $G_{g+1}$, which we also call the {\it tail (dual graph) piece}, the $a_1$ counts the number of vertices minus 1 to denote exclusion of the open dot. In Figure \ref{Ggplus1}, $a_1=6$. If $G_{g+1}$ would only consist of 1 open vertex as in Figure \ref{ex7}, then we say that $a_1=0$ for $G_{g+1}$ and hence $G_g$ is the last piece of that particular dual graph $G$.

\begin{figure}
\centering
\includegraphics[width=135mm, bb=0 0 638 100]{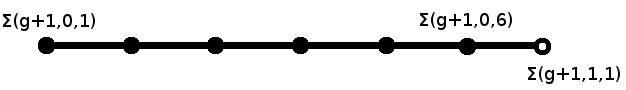}
\caption{Dual graph tail piece $G_{g+1}$}
\label{Ggplus1}
\end{figure}

\begin{defin}
The {\it defining set of data} of a dual graph $G$ is the set of non-negative integers: $g$, $\displaystyle \left\{m_i\right\}_{i=1}^{g+1}$, and $\displaystyle \{a_j^{(i)}\}_{j=1}^{m_i}$. If there is no tail $G_{g+1}$, then we set $m_{g+1}=0$ and $a_1^{(g+1)}=0$.
\end{defin}

\begin{defin}
A {\it modification of the first kind} is the adjoining of a new vertex in the construction of the dual graph which is adjacent to just one older vertex. A {\it modification of the second kind} is the adjoining of a new vertex which is adjacent to two older vertices. In Example \ref{dgex}, adjoining vertex 3 is a modification of the first kind, while adjoining vertex 7 is a modification of the second kind. By convention, the introduction of the first vertex in the first stage of building a dual graph is also considered a modification of the first kind.
\end{defin}

\begin{rem}
In the literature, Favre and Jonsson's definition of {\it free} and {\it satellite} blowups in \cite{fj} is similar to Spivakovsky's modifications of the first and second kind, respectively.
\end{rem}

The dual graphs of divisorial valuations in general are depicted in Figures \ref{div1} and \ref{div2}. Most of the vertices are suppressed for clarity. The sigma label notation will be explained later.

\begin{figure}
\centering
\includegraphics[width=135mm, bb=0 0 638 231]{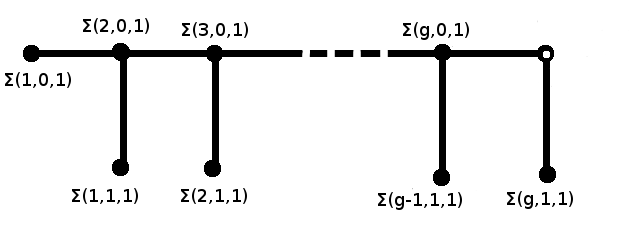}
\caption{Type 0, divisorial valuation, $a_1^{(g+1)}=0$}
\label{div1}
\end{figure}

\begin{figure}
\centering
\includegraphics[width=135mm, bb=0 0 638 231]{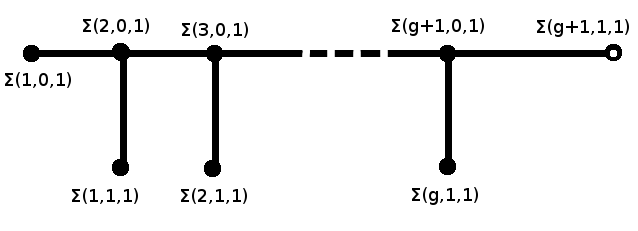}
\caption{Type 0, divisorial valuation, $a_1^{(g+1)}\neq0$}
\label{div2}
\end{figure}

The discussion above is summarized in the following

\begin{defin}
A {\it dual graph} $G$ is a simple connected graph made of {\it dual graph pieces} $G_i$ of the forms depicted in Figures \ref{Gi} and \ref{Ggplus1}, where the vertices are generated by modifications of the first and second kind. The vertices are labeled by $n\in\mathbb N$ and the $n$-th vertex represents the irreducible exceptional component after the $n$-th blowup. Adjacency in the graph represents intersections of exceptional components and the strict transforms of exceptional components. We write $\displaystyle G=\bigcup_i G_i$. If the number of dual graph pieces is finite, then there are $g+1$ pieces if the graph ends with the tail $G_{g+1}$ in Figure \ref{Ggplus1}, else there are $g$ pieces if the graph doesn't end with the tail. In each $G_i$ the horizontal portion is called the {\it odd leg}, and the vertical portion is called the {\it even leg}. The vertices in $G_i$ can be grouped together into $m_i$ segments with consecutively labeled vertices in each segment. We say there are $a_j^{(i)}$ vertices in the $j$-th segment of the $i$-th dual graph piece. Note that the last vertex denoted by the open dot is excluded from the $\displaystyle a_{m_i}^{(i)}$ count in each $G_i$.
\end{defin}

\begin{rem}
Instead of speaking of the $n$-th blowup and labeling the vertices accordingly, Spivakovsky uses different notation and assigns weights to the vertices in the dual graph depending on which type of modification was performed at each step. In addition, Spivakovsky also uses sigma notation to describe the vertices, which we will also adopt.
\end{rem}

\begin{defin} {\it Sigma notation} gives a way of referring to various vertices. Let $\Sigma(i,m,a)$ be the label of the $a$-th vertex in the $(m+1)$-th segment of the $i$-th dual graph piece. We have the following formula:
\begin{equation} \label{sigmasum}
\Sigma(i,m,a) = \sum_{k=1}^{i-1}\sum_{j=1}^{m_k} a_j^{(k)}+\sum_{l=1}^m a_l^{(i)} + a
\end{equation}
where:
\begin{equation*}
\left\{
\begin{array}{l}
1\leq i \leq g+1 \\
0\leq m \leq m_i-1 \\
1\leq a \leq a_{m+1}^{(i)}
\end{array}
\right.
\end{equation*}
and it is understood that $a_1^{(g+1)}$ could be 0.
Notice that the set of all $\Sigma(i,m,a)$ exhausts the labels in the dual graph of a divisorial valuation except for the very last vertex denoted with the open dot.
\end{defin}

\begin{rem}
We will be primarily interested in $\Sigma(i,0,1)$ as well as its predecessor vertex \mbox{$\displaystyle \Sigma\left(i-1,m_{i-1}-1,a_{m_{i-1}}^{(i-1)}\right)$}. The latter is quite cumbersome to write, so the alternative notation $\Sigma(i,0,0)$ will be used to reference it, even though this doesn't follow the rules set in the definition above. Suggestively, $\Sigma(i,0,1) = \Sigma(i,0,0) +1$.
\end{rem}

\begin{ex}
For the dual graph from the opening example of this section, we have: $g=2$, $m_1=2$, $a_1^{(1)}=3$, $a_2^{(1)}=1$, $m_2=2$, $a_1^{(2)}=1$, $a_2^{(2)}=2$, $m_3=0$ and $a_1^{(3)}=0$. The dual graph is shown in Figure \ref{ex8} with sigma notation. Note that $\Sigma(2,0,0)=\Sigma(1,1,1)$ here.
\end{ex}

\begin{figure}
\centering
\includegraphics[width=135mm, bb=0 0 734 378]{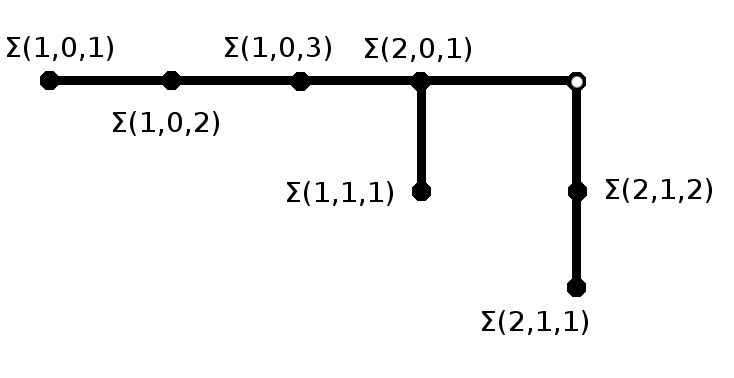}
\caption{The dual graph for the example in sigma notation}
\label{ex8}
\end{figure}

\begin{rem}
The node vertices are of the form $\Sigma(i,0,1)$. The vertex right before $\Sigma(i,0,1)$ is $\Sigma(i,0,0)$. The bottom-most vertex in the even leg of $G_i$ is $\Sigma(i,1,1)$. The right-most vertex in $G_{g+1}$ would not get a label that fits the summation formula (\ref{sigmasum}), but will be labeled $\Sigma(g+1,1,1)$ as a convention to follow the pattern for $\Sigma(i,1,1)$.
\end{rem}

\begin{rem}
The dual graph keeps track of how many of each type of blowup occurred. Notice that $a_1^{(i)}$ counts a $z$-blowup (to get from $\Sigma(i,0,0)$ to $\Sigma(i,0,1)$) followed by a number of consecutive $x$-blowups, where $i\geq2$. All the other $\displaystyle a_{\text{odd}}^{(i)}$ count a number of consecutive $x$-blowups. All $\displaystyle a_{\text{even}}^{(i)}$ count a number of consecutive $y$-blowups. This applies to non-divisorial valuations as well, but some slight changes need to be made.
\end{rem}

In the non-divisorial cases, the number of vertices is infinite. Dual graphs are obtained via modifications of the first and second kind only, so combinatorially we have the following possibilities for dual graphs: Figures \ref{infsingular} to \ref{type4}. Most of the vertices are suppressed in the figures for clarity. The very last vertex denoted by the open dot may not actually be a blowup in the sequence of blowups, but is sometimes inserted into the dual graph for intuition (i.e. in the Type 2 case).

Before continuing, it is useful to establish some additional terminology so we can more easily refer to the various non-divisorial cases. Valuations were classically studied according to the invariants: rank, rational rank $rr$ and dimension $d$ (transcendence degree of the residue field over the base field). In fact, an analysis of Abhyankar's inequality leads to the following classification of valuations in our setting where $(R,\mathfrak{m})$ is a two-dimensional regular local ring, etc. Here the valuation $\nu$ is one of the following cases:
\begin{equation*}
\begin{array}{|c|c|c|c|c|c|} \hline
\text{Rank} & \text{rr} & d & \text{Discreteness} & \text{Value group} & \text{Type} \\ \hline
1 & 1 & 1 & \text{discrete} & \mathbb{Z} & 0 \\ \hline
1 & 1 & 0 & \text{non-discrete} & \text{additive subgroup of}\left.\right. \mathbb{Q}& 1 \\ \hline
1 & 2 & 0 & \text{non-discrete} & \mathbb{Z} + \mathbb{Z} \tau, \left.\right.\text{where} \left. \tau \right. \text{is irrational} & 2 \\ \hline
2 & 2 & 0 & \text{discrete} & \mathbb{Z}^{2} & 3 \left.\right.\text{and} \left.\right. 4.2\\ \hline
1 & 1 & 0 & \text{discrete} & \mathbb{Z} & 4.1 \\ \hline
\end{array}
\end{equation*}
where discreteness refers to the discrete or non-discrete nature of the value groups, and where the value groups are given up to order isomorphism, i.e. an isomorphism that preserves the order.

\begin{rem}
Note the ``Type'' column. The types originate from Spivakovsky's work classifying valuations according to their dual graphs. One difference in notation: we denote Types 4.1 and 4.2 to reflect the rank of the valuation. These two types were originally switched in \cite{spiv}.
\end{rem}

Favre and Jonsson have studied valuations centered on the local ring of formal power series in two complex variables. For geometric intuition, they also considered the interpretations of these valuations when the power series converge at the origin in $\mathbb{C}^2$. This generalizes to smooth points on algebraic surfaces over algebraically closed fields. The following table gives descriptive labels to the types and we will adopt this language in the sequel:
\begin{equation*}
\begin{array}{|c|l|} \hline
\text{Type} & \text{Description} \\ \hline
0 & \text{Divisorial valuation} \\ \hline
1 & \text{Infinitely singular valuation} \\ \hline
2 & \text{Irrational valuation} \\ \hline
3 & \text{Exceptional curve valuation} \\ \hline
4 & \text{Curve valuation} \\ \hline
\end{array}
\end{equation*}

The reader is referred to \cite{fj} for more details. Note that in our usage Type 4 curve valuations fall into two subtypes, Types 4.1 and 4.2.

Now we return to the dual graphs of non-divisorial valuations.

Type 1 infinitely singular valuations are described by: $g=\infty$, $m_i<\infty$ for all $i$. There are infinitely many dual graph pieces $G_i$. See Figure \ref{infsingular}.

\begin{figure}
\centering
\includegraphics[width=135mm, bb=0 0 638 231]{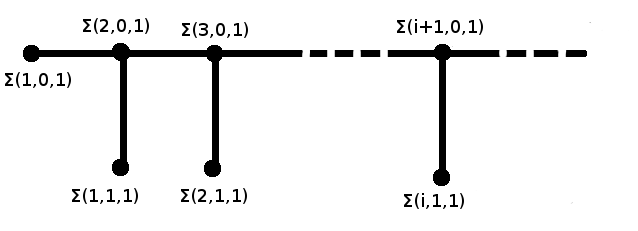}
\caption{Type 1, infinitely singular valuation}
\label{infsingular}
\end{figure}

Type 2 irrational valuations are described by: $g<\infty$, $m_g=\infty$. There are finitely many dual graph pieces, but in $G_g$ the vertices in the infinitely many segments approach the open dot from two sides, never reaching the open dot since it does not correspond to a blowup. See Figure \ref{irrational}. The open dot in the Type 2 case is the limit of where the vertices are heading, so to speak.

\begin{figure}
\centering
\includegraphics[width=135mm, bb=0 0 638 231]{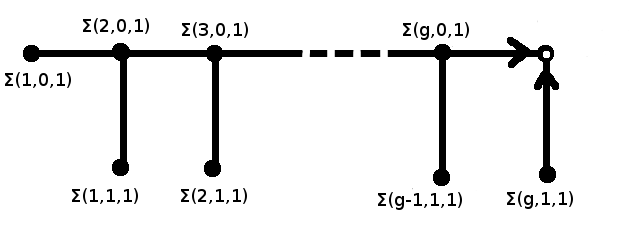}
\caption{Type 2, irrational valuation}
\label{irrational}
\end{figure}

Type 3 exceptional curve valuations are described by: $g<\infty$, $m_g<\infty$, $a_{m_g}=\infty$. There are two subcases, depending on whether $m_g$ is odd or even. In $G_g$, the vertices converge to the open dot from one side only. See Figures \ref{type3odd} and \ref{type3even}. The open dot is an exceptional component in the sequence of blowups.

\begin{figure}
\centering
\includegraphics[width=135mm, bb=0 0 638 231]{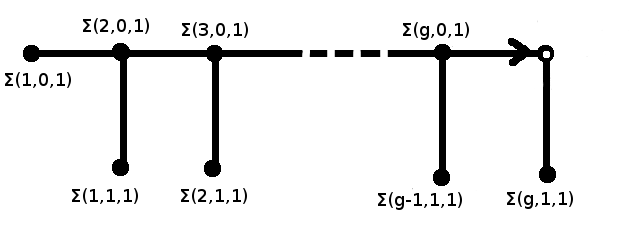}
\caption{Type 3, exceptional curve valuation (odd)}
\label{type3odd}
\end{figure}

\begin{figure}
\centering
\includegraphics[width=135mm, bb=0 0 638 231]{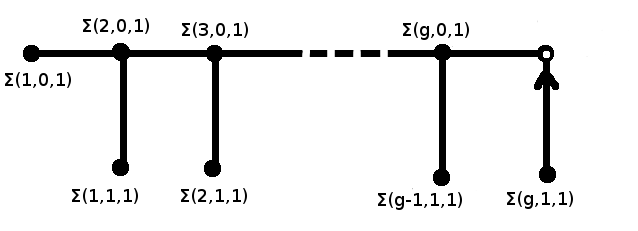}
\caption{Type 3, exceptional curve valuation (even)}
\label{type3even}
\end{figure}

Type 4 curve valuations are described by: $a_1^{(g+1)}=\infty$. The tail $G_{g+1}$ has infinitely many vertices. See Figure \ref{type4}. There are two subcases, Types 4.1 and 4.2, depending on how the blowups in $G_{g+1}$ are interpreted. In Type 4.2, $a_1^{(g+1)}$ encodes 1 $z$-blowup, followed by an infinite number of $x$-blowups. Type 4.1 encodes a mixture of an infinite number of both $x$-blowups and $z$-blowups. This will be discussed further in Section \ref{sectgs}.

\begin{figure}
\centering
\includegraphics[width=135mm, bb=0 0 638 231]{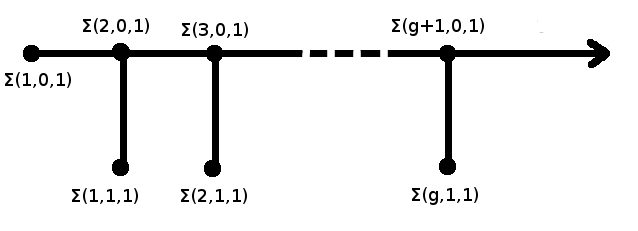}
\caption{Type 4, curve valuation}
\label{type4}
\end{figure}

A given dual graph with its defining set of data is associated with three sets of related continued fractions: $\displaystyle \{\tilde \beta_i\}$, $\displaystyle \{\beta_i'\}$ and $\displaystyle \{\beta_i\}$. These continued fractions will play a vital role in constructing generating sequences from given dual graphs. All the facts about continued fractions that we will use can be found in standard references such as \cite{khin} or \cite{olds}, and the proofs are omitted.

\begin{defin}By a {\it finite continued fraction}, we mean a number of the form:
\begin{equation*}
a_1 + \cfrac{1}{
a_2 + \cfrac{1}{
a_3 + \cfrac{1}{
\ddots + \cfrac{1}{
a_n
}}}}
\end{equation*}
where $a_1\in\mathbb{Z}$, and $a_i\in\mathbb{N} \text{ for } i\geq 2$. A compact notation for this continued fraction is: $\displaystyle [a_1, \ldots, a_n]$. Note that we use $\mathbb{N}$ to denote the positive integers while we use $\mathbb{N}_0$ to denote the non-negative integers.

Similarly, we have {\it infinite continued fractions}: $\displaystyle [a_1, a_2, \ldots]$. Finite continued fractions are rational numbers and infinite continued fractions are irrational numbers.
\end{defin}

\begin{defin}
For a given continued fraction, $\displaystyle [a_1, \ldots, a_n]$ or $\displaystyle [a_1, a_2, \ldots]$, the {\it i-th convergent} is defined to be the fraction $\displaystyle [a_1, \ldots, a_i]$. Let us denote the $i$-th convergent by $\displaystyle \lambda_i/\mu_i$, where $\lambda_i$ and $\mu_i$ are relatively prime.
\end{defin}

\begin{prop} \label{conrec}
Let $\displaystyle \lambda_{-1}=0$, $\displaystyle \lambda_0=1$, $\displaystyle \mu_{-1}=1$ and $\displaystyle \mu_0=0$ by convention. For $i\geq 1$, we have the basic recursive formulas:
\begin{equation*}
\lambda_i=a_{i}\lambda_{i-1}+\lambda_{i-2}
\end{equation*}
\begin{equation*}
\mu_i=a_{i}\mu_{i-1}+\mu_{i-2}
\end{equation*}
\end{prop}

\begin{rem}
There exists a wonderful table method of using these recursive formulas to quickly compute the convergents $\lambda_i/\mu_i$. See pp.24-25 of \cite{olds}.
\end{rem}

\begin{prop} \label{condet} (Determinant Formula)

Given a continued fraction $\displaystyle [a_1, \ldots, a_n]$, we have:
\begin{equation*}
\lambda_n\mu_{n-1}-\lambda_{n-1}\mu_n = (-1)^{n}
\end{equation*}
\end{prop}

\begin{defin}
Let $\displaystyle G=\bigcup G_i$ be a dual graph with its defining set of data. We define the associated continued fractions:

\begin{equation*}
\tilde \beta_i := \left[a_1^{(i)}, a_2^{(i)}, \ldots, a_{m_i}^{(i)}\right]
\end{equation*}

\begin{equation*}
\beta_i' := \left[a_1^{(i)}, a_2^{(i)}, \ldots, a_{m_i}^{(i)}, 1\right]
\end{equation*}
Here $\beta_i'$ is only defined if $\tilde \beta_i$ is rational, i.e. if $m_i < \infty$. Let $\beta_i'=p_i/q_i$. Define $\beta_i$ recursively as follows:

\begin{equation} \label{betarecur}
\left\{
\begin{array}{l}
\beta_0 := \nu(x)
\\\\
\displaystyle \beta_i := q_{i-1}\beta_{i-1} + \frac{1}{q_1 \cdots q_{i-1}}\left(\beta_i'-1\right)\beta_0 \quad \quad \text{for\ } i\geq1
\end{array}
\right.
\end{equation}
\end{defin}

When $\nu$ is rank 1, the $\{\beta_i\}$ will be the values of the generating sequence elements, i.e. $\{\nu(Q_i)\}$, to be defined in Section \ref{sectgs}. Set $q_0=1$ by definition. Note that $\beta_0=1$ for normalized valuations of rank 1. The same recursive formula holds for rank 2 valuations with some small modifications and a different $\beta_0$. See Section \ref{sectgs} and Lemma \ref{genvalue}.

We will be interested in the convergents of $\tilde \beta_i$. Denote the $j$-th convergent by $\lambda_j^{(i)}/\mu_j^{(i)}$, where $1\leq j\leq m_i$. The parentheses superscripts will be suppressed when it is clear from context.

\begin{prop} \label{pq}
Let $\beta_i'=p_i/q_i$. Then,
\begin{equation*}
\left\{
\begin{array}{l}
p_i=\lambda_{m_i} + \lambda_{m_i-1}
\\\\
q_i=\mu_{m_i} + \mu_{m_i-1}
\end{array}
\right.
\end{equation*}
\end{prop}
\begin{proof}
This is a straightforward computation. Notice $\displaystyle \beta_i'=[a_1,\ldots,a_{m_i},1]$ and so $\beta_i'$ shares the same $j$-th convergents with $\tilde \beta_i$ until $j=m_i+1$. Now use Proposition \ref{conrec} to get the desired result.
\end{proof}

We will be interested in the regular parameters after certain blowups and using sigma notation in subscripts is cumbersome for such purposes. The following shorthand will be adopted.

\begin{defin}
Denote by $\left(X_i,Y_i\right)$ the regular parameters after the \mbox{$\Sigma(i,0,1)$} blowup. Denote by $\left(\tilde X_i,\tilde Y_i\right)$ the regular parameters after the \mbox{$\Sigma(i+1,0,0)$} blowup. We will also say that these are the regular parameters at levels \mbox{$\Sigma(i,0,1)$} and \mbox{$\Sigma(i+1,0,0)$}, respectively. Notice that we go from $\left(\tilde X_{i-1},\tilde Y_{i-1}\right)$ to $(X_i,Y_i)$ after one $z$-blowup.
\end{defin}

Lemmas \ref{XY} and \ref{XYtilde} will be useful for computations involving blowups, and hence continued fractions. The following setting will be used in both lemmas.

Let the continued fraction $\beta=[a_1,\ldots,a_n]$ represent a sequence of point blowups. The dual graph will be just one dual graph piece. Let the $i$-th convergent of $\beta$ be denoted $\lambda_i / \mu_i$. Let $(X,Y)$ be the regular parameters before any blowups, i.e. at level 0. Let $\left(\tilde X,\tilde Y\right)$ be the regular parameters after the last blowup, i.e. at level $\sum_{k=1}^n a_k$. Let $(\xi_0,\zeta_0)$ be $(X,Y)$ and let $(\xi_j,\zeta_j)$ be the regular parameters at level $\sum_{k=1}^j a_k$, for  $j\geq 1$. Notice $(\xi_{n},\zeta_{n})$ is just $(\tilde X,\tilde Y)$.

\begin{lem} \label{XY}
We have the following relationships describing $\left(X,Y\right)$ in terms of $\left(\tilde X,\tilde Y\right)$:
\begin{equation*}
\left\{
\begin{array}{llll}
\text{for } n \text{ even,} & X = \tilde X^{\mu_{n-1}} \tilde Y^{\mu_n}  & \text{and} & Y = \tilde X^{\lambda_{n-1}} \tilde Y^{\lambda_n}
\\\\
\text{for } n \text{ odd,} & X = \tilde X^{\mu_n} \tilde Y^{\mu_{n-1}}  & \text{and} & Y= \tilde X^{\lambda_n} \tilde Y^{\lambda_{n-1}}
\end{array}
\right.
\end{equation*}
More generally,
\begin{equation*}
\left\{
\begin{array}{llll}
\text{for } j \text{ even,} & X = \xi_j^{\mu_{j-1}} \zeta_j^{\mu_j}  & \text{and} & Y = \xi_j^{\lambda_{j-1}} \zeta_j^{\lambda_j}
\\\\
\text{for } j \text{ odd,} & X = \xi_j^{\mu_j} \zeta_j^{\mu_{j-1}}  & \text{and} & Y= \xi_j^{\lambda_j} \zeta_j^{\lambda_{j-1}}
\end{array}
\right.
\end{equation*}
\end{lem}

\begin{proof}
First, $X=X^{1}Y^{0}=\xi_0^{\mu_{-1}}\zeta_0^{\mu_0}$ and $Y=X^{0}Y^{1}=\xi_0^{\lambda_{-1}}\zeta_0^{\lambda_0}$. At $\Sigma(1,0,a_1)$, i.e. after $a_1$ $x$-blowups, we get:
\begin{equation*}
X=\xi_1^{a_1\mu_0+\mu_{-1}}\zeta_1^{\mu_0}=\xi_1^{\mu_1}\zeta_1^{\mu_0}
\end{equation*}
and
\begin{equation*}
Y=\xi_1^{a_1\lambda_0+\lambda_{-1}}\zeta_1^{\lambda_0}=\xi_1^{\lambda_1}\zeta_1^{\lambda_0}
\end{equation*}
The rest is an easy induction exercise using Proposition \ref{conrec}.
\end{proof}

\begin{lem} \label{XYtilde}
We have the following relationships describing $\left(\tilde X,\tilde Y\right)$ in terms of $\left(X,Y\right)$:
\begin{equation*}
\left\{
\begin{array}{llll}
\text{for } n \text{ even,} & \tilde X = X^{\lambda_n}/Y^{\mu_n} & \text{and} & \tilde Y = Y^{\mu_{n-1}}/X^{\lambda_{n-1}}
\\\\
\text{for } n \text{ odd,} & \tilde X = X^{\lambda_{n-1}}/Y^{\mu_{n-1}} & \text{and} & \tilde Y = Y^{\mu_n}/X^{\lambda_n}
\end{array}
\right.
\end{equation*}
\end{lem}

\begin{proof}
Assume $n$ is even. $X^{\lambda_n}/Y^{\mu_n}=\tilde X^{\lambda_n\mu_{n-1}-\lambda_{n-1}\mu_n}\tilde Y^{\lambda_n\mu_{n}-\lambda_{n}\mu_n}=\tilde X$ by Lemma \ref{XY} and Proposition \ref{condet}. The other cases are done analogously.
\end{proof}

\begin{rem}In matrix notation, the values of the regular parameters are related as follows:
\begin{equation}\label{matrixvalues}
\left\{
\begin{array}{llll}
\text{for } n \text{ even,} &
\left[\begin{array}{l} \nu(X) \\ \nu(Y) \end{array}\right] =
\left[\begin{array}{ll} \mu_{n-1} & \mu_n \\ \lambda_{n-1} & \lambda_n \end{array}\right]
\left[\begin{array}{l} \nu(\tilde X) \\ \nu(\tilde Y) \end{array}\right]
\\\\
\text{for } n \text{ odd,} &
\left[\begin{array}{l} \nu(X) \\ \nu(Y) \end{array}\right] =
\left[\begin{array}{ll} \mu_n & \mu_{n-1} \\ \lambda_n & \lambda_{n-1} \end{array}\right]
\left[\begin{array}{l} \nu(\tilde X) \\ \nu(\tilde Y) \end{array}\right]
\end{array}
\right.
\end{equation}
\end{rem}

\section{generating sequences} \label{sectgs}

Generating sequences are a tool used in the study of the value semigroup, which in turn encodes information about the resolution of singularities as one of its important applications.

\begin{defin}
The {\it value semigroup} is:
\begin{equation*}
S:=\left\{\nu(x) \left.|\right. x\in \mathfrak{m} \right\}
\end{equation*}
\end{defin}

Note that $S$ is well-ordered since $R$ is Noetherian.

\begin{defin}
Let $I_s:=\{x\in R \left.|\right. \nu(x)\geq s\}$ and $I_s^+:=\{x\in R \left.|\right. \nu(x)> s\}$, where $s\in S$. Contractions of ideals in $V$ to $R$ are called {\it $\nu$-ideals}. The $\{I_s\}$ and $\{I_s^+\}$ are $\nu$-ideals. In fact, $\displaystyle \{I_s\}_{s\in S}$ is the set of all $\nu$-ideals in $R$.
\end{defin}

\begin{defin}
Since $\nu$-ideals are often studied in the context of the {\it associated graded algebra} (over $k$), we define it here:
\begin{equation*}
\text{gr}_{\nu}(R):=\bigoplus_{s\in S} \frac{I_s}{I_s^+}
\end{equation*}
\end{defin}

Geometric interpretations of the associated graded algebra can be found in \cite{tei}.

\begin{defin}
Let $\{Q_i\}_{i=0}^{g'}$ be a (possibly infinite) sequence of elements of $\mathfrak m$. We say that $\{Q_i\}$ is a {\it generating sequence} for $\nu$ if every $\nu$-ideal $I_s \subset R$ is generated by the set:
\begin{equation*}
\left\{\prod_i Q_i^{\alpha_i} \left.\mid\right. \alpha_i \in \mathbb{N}_0, \sum_i \alpha_i \nu(Q_i) \geq s\right\}
\end{equation*}
A {\it minimal generating sequence} is one in which exclusion of any $Q_j$ will cause $\{Q_i\}_{i\neq j}$ to not be a generating sequence.
\end{defin}

In other words, the value semigroup $S$ is given by:
\begin{equation*}
S=\left\{\sum_{i} \alpha_i \nu(Q_i) \left|\right. \alpha_i \in \mathbb{N}_0\right\}
\end{equation*}
and we prefer to think about generating sequences from this perspective. The only discrepancy occurs in the Type 4.1 case. There an infinite number of elements $\displaystyle \{Q_i\}_{i=0}^{\infty}$ is needed to generate the value ideals, but only a finite number of elements $\displaystyle \{Q_i\}_{i=0}^{g}$ is needed to generate the value semigroup. Here $\{Q_i\}_{i=g+1}^{\infty}$ does not generate any new values, yet $\displaystyle \{Q_i\}_{i=g+1}^{\infty}$ is required to be part of the generating sequence as defined above. This will be discussed in greater detail later.

It is known to specialists that minimal generating sequences are of the following form: $Q_0=x, Q_1=y$, and for $i\geq 2$,
\begin{equation}  \label{genrecur}
Q_i=Q_{i-1}^{q_{i-1}}-\sum_{h=1}^{\delta_i} u_h^{(i)} \prod_{j=0}^{i-2} Q_j^{\gamma_{j,h}^{(i)}}
\end{equation}
where $\displaystyle \beta_i'=\frac{p_i}{q_i}$, $\delta_i\in\mathbb{N}$, $\gamma_{j,h}^{(i)}\in\mathbb N_0$, and $0\neq u_h^{(i)}\in k$ such that:
\begin{equation*}
\begin{array}{lll}
\displaystyle \sum_{j=0}^{i-2} \gamma_{j,h}^{(i)}\cdot\nu(Q_j) = q_{i-1}\nu(Q_{i-1}) & , & \text{for } 1\leq h \leq \delta_i
\end{array}
\end{equation*}
Also $\sum u_h\neq 0$ and the $\{u_h\}$ encode information about the centers in the sequence of blowups.

The total number of elements of $\{Q_i\}_{i=0}^{g'}$, i.e. $g'+1$ (possibly infinite), depends on the valuation and can be deduced from the shape of the dual graph. See Section \ref{sectproofs}.

We will provide a new proof that $\{Q_i\}_{i=0}^{g'}$ is a minimal generating sequence for non-divisorial valuations. For our purposes (see Section \ref{sectcr}), we will ignore divisorial valuations. For the sake of clarity, an overview of the proof is given here, but the technical details are left for Section \ref{sectproofs}.

The main idea is very simple: the sequence of blowups will be used to sieve the elements of the value semigroup. $Q_0=x$ and $Q_1=y$ can generate all values of the form: $\displaystyle \alpha_0 \nu(Q_0)+\alpha_1\nu(Q_1)$. Let $S_1:=\{\alpha_0\nu(Q_0)+\alpha_1\nu(Q_1)\}$ and in the general case, let $S_i:=\{\sum_{j=0}^{i} \alpha_j\nu(Q_j)\}$. To generate more elements in $S\setminus S_1$ (or $S\setminus S_i$ in general), a new $Q_2$ (or $Q_{i+1}$ respectively) must be chosen to generate values with new larger denominators or values which increase the rank or rational rank of the values considered thus far in $S_1$ (or $S_i$ in general). These new elements of the generating sequence will be chosen to satisfy various properties according to the dual graph of $\nu$.

The possibility of new generating sequence elements in this process that don't introduce new denominators -- or a rank or rational rank jump -- will be discussed later.

Notice that $x$-blowups involve subtracting the value of the regular parameter $x_i$ to get the value of the next parameter $y_{i+1}$, $\nu(y_{i+1})=\nu(y_i)-\nu(x_i)$, hence an $x$-blowup cannot introduce a new denominator in the values of the regular parameters at the next step. Similarly, $x$-blowups cannot introduce an irrational or increase the rank of the values already sieved. Notice $y$-blowups also have the same limitations. On the other hand, $z$-blowups can introduce new denominators or irrationals or increase the rank since $y_i-cx_i=x_{i+1}y_{i+1}$, where $c$ is the residue of $y_i/x_i$. If $\nu(x_i)=\nu(y_i)$, then $\nu(y_i-cx_i)$ can be greater than $\nu(x_i)$, invoking the ultrametric inequality. This phenomenon can introduce a new denominator (and so forth) in $\nu(y_{i+1})$, so we will call such a $\nu(y_{i+1})$ a {\it jump value}. We naturally turn our attention to the regular parameters at $\Sigma(i,0,0)$ and $\Sigma(i,0,1)$, before and after $z$-blowups. At $\Sigma(2,0,0)$, $Q_2$ must have strict transform $\tilde Y_1-c_2\tilde X_1$ (up to units) and in general at $\Sigma(i,0,0)$, $Q_i$ must have strict transform $\tilde Y_{i-1}-c_i\tilde X_{i-1}$, where $c_i$ is the residue of $\tilde Y_{i-1}/\tilde X_{i-1}$. If $\{Q_i\}$ is defined as in Equation (\ref{genrecur}), then this necessary property of their strict transforms, necessary in order to have jump values, will be shown in Lemma \ref{strict}. Note that the values $\{\nu(Y_i)\}$ are jump values.

If there is a rational rank or rank jump at $\Sigma(i,0,1)$, then we will not have another $z$-blowup available to introduce yet another node $\Sigma(i+1,0,1)$ in the dual graph since we won't be able to get $\nu(\tilde X_i)=\nu(\tilde Y_i)$ via $x$-blowups and $y$-blowups. Only one such value jump can occur in a given dual graph and it must manifest itself in the last dual graph piece.

Let $f\in R$, so $\displaystyle f=x_i^{e_{1,i}}y_i^{e_{2,i}}f^{(i)}$ in $R_i$. It is easy to see that the values of the strict transforms $\nu(f^{(i)})$ decrease as $i$ increases, until the strict transform becomes a unit. This decrease in strict transform values correspond to an increase in the values of the exceptional transforms. There is the logical possibility of many elements in $R$ with the same strict transform $\displaystyle \tilde Y_{i-1} - c_i \tilde X_{i-1}$. We are interested in the ones with the smallest exceptional transform values. The $\{Q_i\}$ must include the minimal valued elements in $R$ that introduce the jump values $\nu(Y_i)=\nu(\tilde Y_{i-1}-c_i\tilde X_{i-1})-\nu(X_i)$ at $\Sigma(i,0,1)$. This minimality property will be shown in Lemma \ref{minimalvalue}.

As alluded to earlier, there is the logical possibility of two or more generating sequence elements, say $Q_i$ and $\overline Q_i$, leading to the introduction of the same new denominator at $\Sigma(i,0,1)$, yet both $Q_i$ and $\overline Q_i$ are necessary to include in a minimal generating sequence. This redundancy is shown to be impossible in Lemma \ref{nonredundant}, if we adopt the viewpoint that the $\{Q_i\}$ should generate the value semigroup instead of the value ideals. Thus, one and only one minimal generating sequence element introduces each new denominator in the process of sieving through the value semigroup.

Similarly, for Types 2, 3 and 4.2, there is the logical possibility of two or more generating sequence elements, say $Q_{g'}$ and $\overline Q_{g'}$, that share the same strict transform at $\Sigma(g',0,0)$, hence opening up the possibility of both being necessary to include in a minimal generating sequence. Lemma \ref{nonredundant234} will show that this redundancy is impossible.

The discussion above is summarized by saying the shape of the dual graph dictates how many elements are in a minimal generating sequence.

Now we need a classical theorem in valuation theory:

\begin{thm}(Abhyankar) \label{lu}
\\
$R$ blows up to become the valuation ring $V$ in the limit of the sequence of point blowups:
\begin{equation*}
V=\bigcup_{i=0}^\infty R_i
\end{equation*}
\end{thm}

\begin{proof}
See Lemma 12 of \cite{ab}.
\end{proof}

Theorem \ref{lu} implies that the sequence of blowups will detect all of the values from $\mathfrak{m}_V$, hence all of the values in the value semigroup will also be reflected in the regular parameters of $\{R_j\}$. The changes in values from $S_{i-1}$ to $S_i$ will show up in the values of $(x_j,y_j)$ for some $j$. In particular, the previous discussion implies that only the parameters at $j=\Sigma(i,0,1)$ will matter, and these in turn come from transforming the $\{Q_i\}$. Thus, the sieving process can be completed by only looking at the $\{Q_i\}$ and hence they form a minimal generating sequence. This is stated as Theorem \ref{minimalgs} below.

Theorem \ref{unique} shows that a generating sequence provides a unique representation of values in the value semigroup.

We now give a description of valuations of Types 0, 3 and 4 are how they are normalized in this paper in preparation for the proofs in Section \ref{sectproofs}.

Let $\nu$ be a Type 0 divisorial valuation with $n$ blowups. $R_n$ will be a discrete valuation ring with uniformizing parameter $x_n$. The last exceptional component $L_n$ will be given locally by $x_n=0$. Let $f\in R$. If $\displaystyle f=x_n^ef^{(n)}$, $x_n\nmid f^{(n)}$, then $\displaystyle \nu(f)=\frac{e}{b}$, where $\displaystyle \nu(x_n)=\frac{1}{b}$ and the value group $\displaystyle \Gamma = \frac{1}{b}\mathbb{Z}$. Note that the valuation was normalized so that $\nu(x)=1$ and Lemmas \ref{valuetilde} and \ref{beta'-1} imply that $\displaystyle b=\prod_{i=1}^{g} q_i$.

Let $\nu$ be a Type 3 exceptional curve valuation. Consider $R_n$ where the integer $\displaystyle n=\Sigma(g,m_g-2,a_{m_g-1})$. There are two cases, depending on whether $m_g$ is odd or even. Normalize $\nu$ as follows. In the odd case, $\nu(x_n)=(0,1)$ and $\nu(y_n)=(1,0)$, and there is an infinite number of $x$-blowups after $n$. In the even case, $\nu(x_n)=(1,0)$ and $\nu(y_n)=(0,1)$, and there is an infinite number of $y$-blowups after $n$. The value group $\Gamma = \mathbb{Z}\times\mathbb{Z}$ in both the odd and even cases. After normalizing at level $n$, the values at level 0 are computed using Formula (\ref{matrixvalues}) (after the last remark in Section \ref{sectdg}) and some basic lemmas from Section \ref{sectproofs}.

Notice the tail dual graph piece $G_{g+1}$ in the Types 4.1 and 4.2 cases cannot admit $y$-blowups since that would force an even leg to show up in $G_{g+1}$. As such, the only blowups available in $G_{g+1}$ are $x$-blowups and $z$-blowups. Both the Type 4.1 and 4.2 cases have a mix of $x$-blowups and $z$-blowups in $G_{g+1}$, but there's an infinite number of $z$-blowups in the Type 4.1 case, whereas there's only one $z$-blowup in $G_{g+1}$ in the Type 4.2 case. There cannot be an infinite number of consecutive $x$-blowups in the Type 4.1 case because there is no rank jump, hence the need for the infinite number of $z$-blowups. Notice the infinite number of $z$-blowups in the Type 4.1 case don't introduce new denominators, or else a $y$-blowup would show up, hence these $z$-blowups don't affect the value semigroup.

In the Type 4.1 case, the mixture of $x$-blowups and $z$-blowups reflect a potential ``analytic change of coordinates.'' As a basic illustrative example, consider the analytic curve given by the power series $\displaystyle y'=\sum_{j=1}^{\infty} c_{e_j}x^{e_j}$, where $c_{e_j}\in k$ and $\{e_j\}$ is a strictly increasing sequence of positive integers. Assume $y'\notin R$. Define:
\begin{equation*}
\begin{array}{l}
Q_0=x \\
Q_1=y \\
\displaystyle Q_2=y-c_{e_1}x^{e_1} \\
\displaystyle Q_3=y-c_{e_1}x^{e_1}-c_{e_2}x^{e_2} \\
\end{array}
\end{equation*}
and in general let $\displaystyle Q_i=y-\sum_{j=1}^{i-1}c_{e_j}x^{e_j}$ for $i\geq 2$. Here $\beta_0=\nu(Q_0)=1$, $\nu(Q_1)=e_1$, $\nu(Q_2)=e_2$ and in general $\beta_i=\nu(Q_i)=e_i$, which comes from the order valuation with respect to $x$. The idea is to let $y$ simulate $y'$ even though $y'$ is not in $R$. Notice:
\begin{equation*}
Q_i=y-\sum_{j=1}^{i-1} c_{e_j}x^{e_j} \neq u x^{e_i}
\end{equation*}
where $u$ is a unit, since $x$ and $y$ are regular parameters. This implies $Q_i$ is needed to generate the $\nu$-ideal $\displaystyle I_{e_i}$. Hence, all of the $\{Q_i\}_{i=0}^{\infty}$ are necessary in a minimal generating sequence; such a generating sequence has an infinite number of elements. We could think of $\nu(Q_i)=(0,e_i)$ and in the limit of blowups we could potentially get $\nu(y')=(1,0)$, a jump in the rank that could also occur if we allowed the analytic change of coordinates from $y$ to $y'$. However, $R$ itself only sees the second non-zero coordinate in values since $y'\notin R$, hence the valuation on $\text{Frac}(R)$ is rank 1. This idea is related to the notion that valuations can ``jump rank'' in the completion of $R$.

The dual graph of the Type 4.1 example just considered would be only the tail piece $G_{g+1}$, where $g=0$. The lack of $y$-blowups here imply that there are no general $\Sigma(i,0,1)$ nodes in the dual graph. The values $\nu(x_i)=1$ for all $i$. We start with $\nu(y)=e_1$, so there are $e_1-1$ $x$-blowups until $\nu(x_{e_1-1})=\nu(y_{e_1-1})=1$. Now a $z$-blowup is performed and we have $\nu(y_{e_1})=e_2-e_1$ (from blowing up $Q_2$), which next leads to $e_2-e_1-1$ $x$-blowups. The sequence of blowups is as follows: $e_1-1$ $x$-blowups, a $z$-blowup, $e_2-e_1-1$ $x$-blowups, a $z$-blowup, $e_3-e_2-1$ $x$-blowups, a $z$-blowup, $e_4-e_3-1$ $x$-blowups, a $z$-blowup, and so forth. This is easily seen by applying the sequence of blowups above to the set $\{Q_i\}$.

The phenomenon noted above can be shifted to $\Sigma(g+1,0,1)$ to yield the fact that minimal generating sequences for more general Type 4.1 dual graphs can also have an infinite number of elements. The tail piece $G_{g+1}$ is the crucial part. Familiarity with the contents of Section \ref{sectproofs} will be helpful for the following arguments. It might even be wise to read the following arguments lightly at first, and return here after Section \ref{sectproofs} has been read.

Now in this setting $\displaystyle y'=\sum_{j=1}^{\infty} c_{e_j}X_{g+1}^{e_j}$, and $Y_{g+1}$ is used to simulate $y'$. Suppose $\displaystyle q_g \nu(Q_g)=\frac{n}{q_1\cdots q_g}$. To get $Y_{g+1}$ into play, the proof of Lemma \ref{genvalue} implies
\begin{equation*}
Q_{g+1}=X_{g+1}^nY_{g+1}u
\end{equation*}
where
\begin{equation*}
Q_{g+1}:=Q_{g}^{q_g}-T_1
\end{equation*}
and
\begin{equation*}
T_1:=\sum_{h} u_{h,1} \prod_{j=0}^{g-1}Q_j^{\gamma_{j,h}^{(1)}}
\end{equation*}
such that
\begin{equation*}
\sum_{j} \gamma_{j,h}^{(1)}\cdot \nu(Q_j)=q_g \nu(Q_g)=\frac{n}{q_1\cdots q_g}.
\end{equation*}
Here $u$ is a unit in $\displaystyle R_{\Sigma(g+1,0,1)}$, the $\{u_{h,1}\}\in k$, and $\sum u_{h,1}\neq 0$. Define:
\begin{equation*}
Q_{g+i}:=Q_g^{q_g}-\sum_{j=1}^i T_j \text{ , for } i\geq 1
\end{equation*}
where
\begin{equation*}
T_i:=\sum_{h} u_{h,i} \prod_{j=0}^{g-1}Q_j^{\gamma_{j,h}^{(i)}}
\end{equation*}
such that
\begin{equation*}
\sum_{j} \gamma_{j,h}^{(i)}\cdot \nu(Q_j)=\frac{n+e_{i-1}}{q_1\cdots q_g} \text{ , for } i\geq 1.
\end{equation*}
where $e_0=0$ by convention. The $\{u_{h,i}\}\in k$ and furthermore, for $i\geq 2$, we have $0\neq \sum_h u_{h,i}\cong uc_{e_{i-1}}$ at $\Sigma(g+1,0,1)$. The motivation comes from Lemma \ref{leveli+}: $T_i=X_{g+1}^{n+e_{i-1}}\left(\sum_h u_{h,i}\right)$. Factoring out $uX_{g+1}^{n}$ allows for the mimicking of $y'$ at $\Sigma(g+1,0,1)$ as was done earlier.

In other words,
\begin{equation*}
\begin{array}{rl}
\displaystyle Q_{g+2} & \displaystyle =X_{g+1}^{n}Y_{g+1}u-X_{g+1}^{n+e_1}uc_{e_1}\\\\
& \displaystyle =uX_{g+1}^{n}\left(Y_{g+1}-c_{e_1}X_{g+1}^{e_1}\right)
\end{array}
\end{equation*}
where the $X_{g+1}^{n}$ is needed to reach level $\Sigma(g+1,0,1)$. Factoring out $X_{g+1}^{n}$ allows for the setup in the simpler example.

At level $\Sigma(g+1,0,1)$,
\begin{equation*}
\begin{array}{rl}
\displaystyle Q_{g+i} & \displaystyle = X_{g+1}^{n}Y_{g+1}u-\sum_{j=1}^{i-1} X_{g+1}^{n+e_{j}}uc_{e_{j}} \\
& \displaystyle = uX_{g+1}^n\left(Y_{g+1}-\sum_{j=1}^{i-1} c_{e_j}X_{g+1}^{e_{j}}\right)
\end{array}
\end{equation*}

For $i\geq 1$,
\begin{equation*}
\beta_{g+i}=\nu(Q_{g+i})=\frac{n+e_i}{q_1\cdots q_g}
\end{equation*}
It is an easy exercise to see that $\displaystyle \{Q_i\}_{i=0}^{\infty}$ are all necessary in a minimal generating sequence using the fact that $X_{g+1}$ and $Y_{g+1}$ are regular parameters and then essentially applying the same argument as in the earlier simpler case. Notice the value group $\displaystyle \Gamma=\frac{1}{q_1\cdots q_g}\mathbb{Z}$.

For Type 4.1 valuations, although $\{Q_i\}_{i=0}^{\infty}$ form a minimal generating sequence viewed from the perspective of value ideals, only $\{Q_i\}_{i=0}^g$ is necessary to generate the value semigroup.

Lastly, the Type 4.2 valuations are normalized with $\beta_0=\nu(x)=(0,1)$. At $\Sigma(g+1,0,1)$,
\begin{equation*}
\nu(X_{g+1})=\left(0,\frac{1}{q_1\cdots q_g}\right)
\end{equation*}
and
\begin{equation*}
\nu(Y_{g+1})=\left(1,\frac{n}{q_1\cdots q_g}\right)
\end{equation*}
where $n\in\mathbb{Z}$. Notice the second coordinate can be negative since the first coordinate is positive, hence this value is in the value semigroup $S$.

By Lemma \ref{genvalue}, $\nu(Q_{g+1})=q_g \nu(Q_g)+\nu(Y_{g+1})$ hence
\begin{equation*}
\nu(Q_{g+1})=\left(1,\frac{n'}{q_1\cdots q_g}\right)
\end{equation*}
where $n'\in\mathbb{Z}$. The value group $\displaystyle \Gamma=\mathbb{Z}\times\frac{1}{q_1\cdots q_g}\mathbb{Z}$.

In $G_{g+1}$, the Type 4.2 case has one $z$-blowup to go from $\Sigma(g+1,0,0)$ to $\Sigma(g+1,0,1)$ followed by an infinite number of $x$-blowups. Notice the rank jump forces the infinite number of $x$-blowups since it is always true that
\begin{equation*}
\left(1,\frac{n}{q_1\cdots q_g}\right)>\left(0,\frac{1}{q_1\cdots q_g}\right)
\end{equation*}
for any $n \in\mathbb{Z}$.

\section{proofs} \label{sectproofs}

The notation used here is the same as in the previous sections. This section will flesh out the proof of minimal generating sequences outlined in Section \ref{sectgs}.

A minimal generating sequence $\{Q_i\}_{i=0}^{g'}$ has $g'+1$ elements. The last index $g'$ depends on the dual graph of $\nu$ and it will be shown that:
\begin{equation*}
\begin{array}{|c|c|} \hline
\text{Type} & g'\\ \hline
1 & \infty \\ \hline
2 & g \\ \hline
3 & g \\ \hline
4.1 & g \\ \hline
4.2 & g+1 \\ \hline
\end{array}
\end{equation*}
We will take $g'$ to be the respective values shown in the table above to make the arguments cleaner. This choice will be justified later in the proof of Lemma \ref{genvalue} when establishing minimal generating sequences from given dual graphs.

\begin{rem}
  For divisorial valuations (Type 0), $g'=g$ if $a_1^{(g+1)}=0$, and $g'=g+1$ if $a_1^{(g+1)}\neq 0$. If we adopt the viewpoint that generating sequences should generate the value semigroup instead of the value ideals, then $g'=g$ in both cases. However, for some applications such as computing Poincar\'e series, it doesn't make sense to exclude the last $Q_{g+1}$ in the $a_{1}^{(g+1)}\neq 0$ case. We will restrict our attention to the non-divisorial valuations and mention this only for completeness.
\end{rem}

\begin{lem} \label{valuetilde} At $\displaystyle \Sigma(i,0,0)$, where $2 \leq i \leq g'$:
\begin{equation*}
\left\{
\begin{array}{l}
\displaystyle \nu(\tilde X_{i-1})=\frac{1}{q_1\cdots q_{i-1}}\beta_0
\\\\
\displaystyle \nu(\tilde Y_{i-1})=\frac{1}{q_1\cdots q_{i-1}}\beta_0
\end{array}
\right.
\end{equation*}
\end{lem}

\begin{proof}
  At $\displaystyle \Sigma(i,0,0)$, it is always the case that $\nu(\tilde X_{i-1}) = \nu(\tilde Y_{i-1})$, which is what makes the next $z$-blowup possible here, so we only have to prove the result for $\nu(\tilde X_{i-1})$. Note that $\nu(x)=\beta_0$. Use induction on $i$. By Lemma \ref{XY}, $\nu(x)=\nu\left(\tilde X_{1}^{\mu_{m_1}+\mu_{m_1-1}}\right)$ in both the odd and even $m_1$ cases since $\nu(\tilde X_{1})=\nu(\tilde Y_{1})$. By Proposition \ref{pq}, $\mu_{m_1}+\mu_{m_1-1}=q_1$ and the base case is done. Assume the result is true up to $\Sigma(i-1,0,0)$. By Lemma \ref{XY}, $\nu(\tilde X_{i-2})=\nu\left(\tilde X_{i-1}^{\mu_{m_{i-1}}+\mu_{m_{i-1}-1}}\right)$ in both the odd and even $m_{i-1}$ cases. Here we set $\beta=\tilde \beta_{i-1}$, $X=\tilde X_{i-2}$, $Y=\tilde Y_{i-2}-c_{i-1}\tilde X_{i-2}$, $\tilde X=\tilde X_{i-1}$ and $\tilde Y=\tilde Y_{i-1}$ in applying Lemma \ref{XY}. Note that $\mu_{m_{i-1}}+\mu_{m_{i-1}-1}=q_{i-1}$ by Proposition \ref{pq}. Thus, $\nu(\tilde X_{i-2})=\nu\left(\tilde X_{i-1}^{q_{i-1}}\right)$ and the proof is complete using the inductive hypothesis.
\end{proof}

\begin{lem} \label{beta'-1} If $\nu$ is not Type 2, 3 or 4.2, then at $\Sigma(i,0,1)$, where $2\leq i \leq g'$:
\begin{equation*}
\left\{
\begin{array}{l}
\displaystyle \nu(X_i)=\frac{1}{q_1\cdots q_{i-1}}\beta_0
\\\\
\displaystyle \nu(Y_i)=\frac{1}{q_1\cdots q_{i-1}}\left(\beta_i'-1\right)\beta_0
\end{array}
\right.
\end{equation*}
If $\nu$ is Type 2 or 3, then the formulas hold except for $\nu(Y_g)$. If $\nu$ is Type 4.2, then the formulas hold except for $\nu(Y_{g+1})$. See Lemma \ref{genvalue}.
\end{lem}
\begin{proof}
Notice $\nu(X_i)=\nu(\tilde X_{i-1})$, so $\nu(X_i)$ is done by Lemma \ref{valuetilde}.

For $\nu(Y_i)$, first observe that $\tilde Y_{i-1}-c_i\tilde X_{i-1}$ will transform to $X_iY_i$ at level $\Sigma(i,0,1)$, where $c_i\in k$ is the residue of $\tilde Y_{i-1}/\tilde X_{i-1}$. We just need to find the value of $\tilde Y_{i-1}-c_i\tilde X_{i-1}$ and subtract $\nu(X_i)=\frac{1}{q_1\cdots q_{i-1}}\beta_0$. Now apply Lemma \ref{XY}, setting $\beta=\tilde \beta_i$, setting $X=\tilde X_{i-1}$ and $Y=\tilde Y_{i-1}-c_i\tilde X_{i-1}$, and setting $\tilde X = \tilde X_i$ and $\tilde Y= \tilde Y_i$. The two sets of parameters $(X,Y)$ and $(\tilde X,\tilde Y)$ are related by the blowups encoded by $\tilde \beta_i$. By Lemmas \ref{XY} and \ref{valuetilde}, we get $\displaystyle \nu(Y)=\frac{\lambda_{m_i}+\lambda_{m_i-1}}{q_1\cdots q_i}\beta_0$ in both the odd and even $m_i$ cases. Using Proposition \ref{pq},
\begin{equation*}
\nu\left(\tilde Y_{i-1}-c_i\tilde X_{i-1}\right)=\nu(Y)=\frac{p_i}{q_1\cdots q_i}\beta_0=\frac{1}{q_1\cdots q_{i-1}} \beta_i' \beta_0
\end{equation*}
and the proof is complete.
\end{proof}

\begin{lem} \label{leveli}
For $2\leq i\leq g'$, $Q_i$ transforms to:
\begin{equation*}
Q_i=X_{i}^eY_{i}u
\end{equation*}
for some $e\in\mathbb{N}$ and where $u$ is a unit.
\end{lem}
\begin{proof}
This is an easy corollary of Lemma \ref{strict}. Apply a $z$-blowup.
\end{proof}

\begin{lem} \label{leveli+}
For $0\leq i\leq g'-1$, $Q_i$ transforms to:
\begin{equation*}
Q_i=X_{j}^eu
\end{equation*}
for some $e\in\mathbb{N}$, and where $i+1\leq j\leq g'$, and $u$ is a unit.
\end{lem}

\begin{proof}
This will be proved in the proof of Lemma \ref{strict}.
\end{proof}
\begin{rem}
Note that Lemma \ref{leveli} shows what happens to $Q_i$ at $\Sigma(i,0,1)$ while Lemma \ref{leveli+} shows what happens to $Q_i$ afterward, at $\Sigma(j,0,1)$ for $j\geq i+1$.
\end{rem}

\begin{lem} \label{strict}
For $2\leq i\leq g'$, $Q_i$ transforms to:
\begin{equation*}
Q_i=\tilde X_{i-1}^{f_1}\tilde Y_{i-1}^{f_2}\left(\tilde Y_{i-1}-c_i\tilde X_{i-1}\right)u
\end{equation*}
for some $f_1,f_2\in\mathbb{N}$, and where $u$ is a unit and $c_i\in k$ is the residue of $\tilde Y_{i-1}/\tilde X_{i-1}$.
\end{lem}

\begin{proof}
First, we show $\displaystyle \nu(Q_1)=\frac{p_1}{q_1}\beta_0$. Let $m=m_1$. By Lemma \ref{XYtilde} and the fact that $\nu(\tilde X_1)=\nu(\tilde Y_1)$, we get:
\begin{equation*}
\nu\left(\frac{x^{\lambda_m}}{y^{\mu_m}}\right)=\nu\left(\frac{y^{\mu_{m-1}}}{x^{\lambda_{m-1}}}\right)
\end{equation*}
or
\begin{equation*}
\nu\left(\frac{x^{\lambda_{m-1}}}{y^{\mu_{m-1}}}\right)=\nu\left(\frac{y^{\mu_{m}}}{x^{\lambda_{m}}}\right)
\end{equation*}
depending on whether $m$ is even or odd, respectively. In either case,
\begin{equation*}
\left(\lambda_m+\lambda_{m-1}\right)\nu(x)=\left(\mu_m+\mu_{m-1}\right)\nu(y)
\end{equation*}
By Proposition \ref{pq}, $\lambda_m+\lambda_{m-1}=p_1$ and $\mu_m+\mu_{m-1}=q_1$, so $\displaystyle \nu(Q_1)=\frac{p_1}{q_1}\beta_0$.

Now, use induction on $i$ to prove the lemma. By Equation (\ref{genrecur}) in Section \ref{sectgs},
\begin{equation*}
Q_2=Q_1^{q_1}-u_1Q_0^{\gamma_{0,1}}
\end{equation*}
and $\displaystyle q_1\nu(Q_1) = \gamma_{0,1} \beta_0$. Thus, $\gamma_{0,1}=p_1$ and $\displaystyle Q_2 = y^{q_1} - u_1x^{p_1}$.

Assume $m=m_1$ is even for concreteness and apply Lemma \ref{XY}:
\begin{equation*}
Q_2 = \tilde X_1^{q_1\lambda_{m-1}} \tilde Y_1^{q_1\lambda_{m}} - u_1 \tilde X_1^{p_1\mu_{m-1}} \tilde Y_1^{p_1\mu_{m}}
\end{equation*}
Substituting $q_1 = \mu_m + \mu_{m-1}$ and $p_1=\lambda_m + \lambda_{m-1}$, then factoring, we get:
\begin{equation*}
Q_2 = \tilde X_1^{f_1} \tilde Y_1^{f_2}\left(\tilde Y_1^{\lambda_m\mu_{m-1}-\lambda_{m-1}\mu_m} -u_1 \tilde X_1^{\lambda_m\mu_{m-1}-\lambda_{m-1}\mu_m} \right)
\end{equation*}
where $f_1=\lambda_{m-1}\mu_m+\lambda_{m-1}\mu_{m-1}$ and $f_2=\lambda_{m-1}\mu_m+\lambda_m\mu_m$.
Now apply Proposition \ref{condet} to simplify the inside of the parentheses.
\begin{equation*}
Q_2 = \tilde X_1^{f_1} \tilde Y_1^{f_2}\left(\tilde Y_1 -u_1 \tilde X_1 \right)
\end{equation*}

This shows the base step for the induction in the even $m_1$ case. The odd $m_1$ case is similar and the details are omitted. Now we show $Q_{i+1}$ behaves nicely given the inductive hypothesis up to level $i$. The plan of attack is to compute the total transforms of the $\{Q_j\}_{j=0}^{i}$ at $\Sigma(i,0,1)$, then use these calculations to prove the conclusion for $Q_{i+1}$.

For $Q_0=x$ and $Q_1=y$, applying Lemma \ref{XY} followed by a $z$-blowup and then repeating the process shows that $Q_0=X_j^{f_{0,j}}u_{0,j}$ and $Q_1=X_j^{f_{1,j}}u_{1,j}$ at $\Sigma(j,0,1)$ for $2\leq j \leq g'$, where $u_{0,j}$ and $u_{1,j}$ are units, and $f_{0,j},f_{1,j}\in\mathbb{N}$.

For $2\leq j \leq i$, at level $\Sigma(j,0,0)$, assume:
\begin{equation*}
Q_j=\tilde X_{j-1}^{f_1} \tilde Y_{j-1}^{f_2} \left(\tilde Y_{j-1} - c_j \tilde X_{j-1}\right) u
\end{equation*}
where the positive integers $f_1$ and $f_2$ are understood to be different for each $j$, but we suppress any subscripts to indicate as such for the sake of lucidity. Although the units $u$ might vary with each blowup, they will remain units, so we suppress notation here as well.

Performing the next $z$-blowup to get to level $\Sigma(j,0,1)$:

\begin{equation*}
\begin{array}{rl}
\displaystyle Q_j= & X_{j}^{f_1}X_{j}^{f_2}\left(Y_{j}+c_j\right)^{f_2}X_{j}Y_{j} u
\\\\
\displaystyle = & X_{j}^{f_1+f_2+1}Y_{j} u
\end{array}
\end{equation*}
where the $\left(Y_{j}+c_j\right)^{f_2}$ was absorbed into the unit. Let $f_3=f_1+f_2+1$ for simplicity.

Let $m=m_j$ and assume $m_j$ is odd for concreteness. The even case is similar. Tracing blowups to levels $\Sigma(j+1,0,0)$ and $\Sigma(j+1,0,1)$, we have:
\begin{equation*}
\begin{array}{rl}
\displaystyle Q_j=&\tilde X_{j}^{f_3\mu_{m}} \tilde Y_{j}^{f_3\mu_{m-1}} \tilde X_{j}^{\lambda_{m}} \tilde Y_{j}^{\lambda_{m-1}}u
\\\\
\displaystyle =&\tilde X_{j}^{f_3\mu_{m}+\lambda_m} \tilde Y_{j}^{f_3\mu_{m-1}+\lambda_{m-1}}u
\\\\
\displaystyle =& X_{j+1}^{f_3(\mu_{m}+\mu_{m-1})+\lambda_m+\lambda_{m-1}} (Y_{j+1}+c_{j+1})^{f_3\mu_{m-1}+\lambda_{m-1}}u
\\\\
\displaystyle =& X_{j+1}^{f_4}u
\end{array}
\end{equation*}
where the $(Y_{j+1}+c_{j+1})^{f_3\mu_{m-1}+\lambda_{m-1}}$ was absorbed into $u$ and the exponent of $X_{j+1}$ was relabeled as $f_4$.

Tracing the blowups further, it is easy to see that $Q_j$ will have a total transform of the form $\displaystyle X_{k}^{e_k}u$ at level $\Sigma(k,0,1)$ for all $j+1\leq k \leq g'$, where $e_k\in\mathbb{N}$ and $u$ is a unit. This proves Lemma \ref{leveli+} once the proof of Lemma \ref{strict} is complete.

Now, we return to $\displaystyle Q_{i+1}=Q_i^{q_i}-\sum_{h=1}^{\delta_{i+1}} u_h \prod_{j=0}^{i-1}Q_j^{\gamma_{j,h}}$. By the previous discussion, at level $\Sigma(i,0,1)$, $Q_i=X_{i}^{f_1}Y_{i}u$ for some $f_1\in\mathbb{N}$, and so $Q_i^{q_i}=X_i^{q_if_1}Y_i^{q_i}u$, ignoring changes to the unit. Also by the previous discussion, each of the $\prod Q_j^{\gamma_{j,h}}$ will transform to $X_{i}^{f_2}u_h$ for all $h$, with the same power $f_2$ and differing only in the units $u_h$; just transform each $Q_j^{\gamma_{j,h}}$, then collect the $X_{i}$ factors together. The common value $\nu(\prod Q_j^{\gamma_{j,h}})=q_i\nu(Q_i)$ ensures the same $f_2$ for all $h$. Now, factor out the $X_{i}^{f_2}$ in the total transform of $\sum u_h \prod Q_j^{\gamma_{j,h}}$ and set $u'=\sum_{h=1}^{\delta_{i+1}} u_h$. We get:
\begin{equation} \label{qif1}
Q_{i+1}= X_{i}^{q_if_1}Y_{i}^{q_i}u - X_{i}^{f_2}u'
\end{equation}

Notice that $u'\neq 0$ or else $Q_{i+1}$ would not give a jump in value. Note that the two terms on the right both have value $q_i \nu(Q_i)$. Using Lemma \ref{beta'-1} and taking values we have:
\begin{equation*}
\nu\left(X_{i}^{q_if_1}Y_{i}^{q_i}u\right)=q_if_1 \cdot \frac{1}{q_1\cdots q_{i-1}}\beta_0 + q_i\cdot \frac{1}{q_1\cdots q_{i-1}}(\beta_i'-1)\beta_0
\end{equation*}
\begin{equation*}
\nu(X_i^{f_2}u')=f_2\cdot \frac{1}{q_1 \cdots q_{i-1}}\beta_0
\end{equation*}
Set equal and clear the $\prod q_j$ and $\beta_0$. Since $\beta_i'=p_i/q_i$, we get:
\begin{equation*}
q_i f_1=f_2-p_i+q_i
\end{equation*}
Substituting for $q_if_1$ in (\ref{qif1}), we have:
\begin{equation*}
\begin{array}{rl}
\displaystyle Q_{i+1}= & X_{i}^{f_2-p_i+q_i}Y_{i}^{q_i}u - X_{i}^{f_2}u'
\\\\
= & X_{i}^{f_2-p_i+q_i}\left[Y_{i}^{q_i} - c_{i+1}X_{i}^{p_i-q_i}\right]u
\end{array}
\end{equation*}
where we factored out $u$ and set $c_{i+1}\cong u'/u$, where $c_{i+1}\in k$. This is possible because the ring $R_{\Sigma(i,0,1)}$ is localized at its maximal ideal and the residue field is isomorphic to $k$.

Notice it is only necessary to show that $Y_{i}^{q_i} - c_{i+1}X_{i}^{p_i-q_i}$ transforms to the form: $\tilde X_{i}^{f_3} \tilde Y_{i}^{f_4} \left(\tilde Y_{i}-c_{i+1} \tilde X_{i}\right)$, where $f_3,f_4\in\mathbb{N}$. Let $m=m_i$ and assume it is even for concreteness. The odd case is similar. Apply Lemma \ref{XY} using $X=X_i$, $Y=Y_i$, $\beta=\left[a_1^{(i)}-1,a_2^{(i)},\cdots,a_{m}^{(i)}\right]$. The subtraction of 1 in $a_1^{(i)}-1$ represents one less blowup since the $(X_i,Y_i)$ parameters occur after a $z$-blowup, which we have to account for. Let $\lambda_j/\mu_j$ be the convergents of $\beta_i'$. Then the convergents of $\beta$ will be $\lambda_j/\mu_j-1=(\lambda_j-\mu_j)/\mu_j$. Applying Lemma \ref{XY} gives:
\begin{equation*}
Y_{i}^{q_i} - c_{i+1}X_{i}^{p_i-q_i} = \tilde X_i^{q_i(\lambda_{m-1}-\mu_{m-1})} \tilde Y_i^{q_i(\lambda_m-\mu_m)} - c_{i+1} \tilde X_i^{(p_i-q_i)\mu_{m-1}} \tilde Y_i^{(p_i-q_i)\mu_m}
\end{equation*}
Expanding the exponents, we get:
\begin{equation*}
\begin{array}{rl}
q_i(\lambda_{m-1}-\mu_{m-1}) & = (\mu_{m-1}+\mu_m)(\lambda_{m-1}-\mu_{m-1})
\\\\
 & =\lambda_{m-1}\mu_{m-1}+\lambda_{m-1}\mu_m-\mu_{m-1}^2-\mu_{m-1}\mu_m
\end{array}
\end{equation*}
\begin{equation*}
\begin{array}{rl}
q_i(\lambda_m-\mu_m) & = (\mu_{m-1}+\mu_m)(\lambda_m-\mu_m)
\\\\
& =\lambda_m\mu_{m-1}+\lambda_m\mu_m-\mu_{m-1}\mu_m-\mu_m^2
\end{array}
\end{equation*}
\begin{equation*}
\begin{array}{rl}
(p_i-q_i)\mu_{m-1} & = (\lambda_{m-1}+\lambda_m-\mu_{m-1}-\mu_m)\mu_{m-1}
\\\\
& =\lambda_{m-1}\mu_{m-1}+\lambda_m\mu_{m-1}-\mu_{m-1}^2-\mu_{m-1}\mu_m
\end{array}
\end{equation*}
\begin{equation*}
\begin{array}{rl}
(p_i-q_i)\mu_m & =(\lambda_{m-1}+\lambda_m-\mu_{m-1}-\mu_m)\mu_m
\\\\
& =\lambda_{m-1}\mu_m+\lambda_m\mu_m-\mu_{m-1}\mu_m-\mu_m^2
\end{array}
\end{equation*}

Let:
\begin{equation*}
f_3=\lambda_{m-1}\mu_{m-1}+\lambda_{m-1}\mu_m-\mu_{m-1}^2-\mu_{m-1}\mu_m
\end{equation*}
and
\begin{equation*}
f_4=\lambda_{m-1}\mu_m+\lambda_m\mu_m-\mu_{m-1}\mu_m-\mu_m^2
\end{equation*}
Factor out $\tilde X^{f_3}\tilde Y^{f_4}$:
\begin{equation*}
Y_{i}^{q_i} - c_{i+1}X_{i}^{p_i-q_i} = \tilde X_i^{f_3} \tilde Y_i^{f_4}\left(\tilde Y_i^{\lambda_m \mu_{m-1} - \lambda_{m-1} \mu_m} - c_{i+1}\tilde X_i^{\lambda_m\mu_{m-1}-\lambda_{m-1}\mu_m}\right)
\end{equation*}
By Proposition \ref{condet}, $\lambda_m\mu_{m-1}-\lambda_{m-1}\mu_m=1$ and we are done.
\end{proof}

\begin{lem} \label{minimalvalue}
For $2\leq i\leq g'$, $Q_i$ is a lowest valued element of $R$ whose strict transform at $\Sigma(i,0,0)$ has the form: $\tilde Y_{i-1}-c_i\tilde X_{i-1}$.
\end{lem}

\begin{proof}
The proof is by induction on $i$. All units are dropped for clarity. At $\Sigma(2,0,0)$, the total transform of a potential minimal generating sequence element $\bar Q_2$ after $Q_0=x$ and $Q_1=y$ is of the form:
\begin{equation*}
\tilde X_1^{f_1}\tilde Y_1^{f_2}\left(\tilde Y_1 - \tilde X_1\right)
\end{equation*}
Consider the exceptional and strict transforms of $\bar Q_2$ at $\Sigma(1,0,1)$. Notice the transformation of the exceptional transform at $\Sigma(1,0,1)$ to $\Sigma(2,0,0)$  only increases the $f_1$ and $f_2$, but does not affect the strict transform $\tilde Y_1 - \tilde X_1$. Thus, the plan of attack is to consider what strict transform at $\Sigma(1,0,1)$ (and at $\Sigma(i,0,1)$ in the general case) can have such a total transform at $\Sigma(2,0,0)$ (and at $\Sigma(i+1,0,0)$, respectively). We work with $\Sigma(1,0,1)=1$ rather than $\Sigma(1,0,0)=0$ so that the arguments in the base step can be carried over to the inductive step without much modification.

Up to units, the strict transform of $\bar Q_2$ at $\Sigma(1,0,1)$ is of the form $Y_1^{e_1}-X_1^{e_2}$, for some $e_1,e_2\in\mathbb{N}$. Two equal-valued terms are needed to have a jump in value via the ultrametric inequality. Technically, three (or more) terms could lead to a value jump, but all three (or more) terms would have to share the same value, say $\varepsilon$, and so using only two terms would guarantee the minimal such $\varepsilon$. The two terms have no common factors because any such common factor would instead show up in the exceptional transform of $\bar Q_2$ at $\Sigma(1,0,1)$, hence the two terms are of the form $Y_1^{e_1}$ and $X_1^{e_2}$.

As in the proof of Lemma \ref{strict}, we use $\beta_1'-1$ in the transformation formulas (Lemma \ref{XY}) to account for one $x$-blowup from level 0 to level 1:
\begin{equation*}
\begin{array}{l}
X_1 = \tilde X_1^{\mu_{m-1}}\tilde Y_1^{\mu_m} \\
Y_1 = \tilde X_1^{\lambda_{m-1}-\mu_{m-1}}\tilde Y_1^{\lambda_m-\mu_m}
\end{array}
\end{equation*}
where $m=m_1$ is assumed to be even; the odd case is similar.
\begin{equation*}
Y_1^{e_1}-X_1^{e_2}=\tilde X_1^{e_1(\lambda_{m-1}-\mu_{m-1})}\tilde Y_1^{e_1(\lambda_m-\mu_m)} - \tilde X_1^{e_2\mu_{m-1}}\tilde Y_1^{e_2\mu_m}
\end{equation*}
Also:
\begin{equation*}
Y_1^{e_1}-X_1^{e_1} = \tilde X_1^{f_1}\tilde Y_1^{f_2+1} - \tilde X_1^{f_1+1}\tilde Y_1^{f_2}
\end{equation*}
Equating the exponents, we get the system of equations:
\begin{equation*}
\begin{array}{l}
e_1(\lambda_{m-1}-\mu_{m-1}) = f_1 \\
e_1(\lambda_m-\mu_m) = f_2+1 \\
e_2\mu_{m-1} = f_1+1 \\
e_2\mu_m = f_2
\end{array}
\end{equation*}
Eliminating $f_1$ and $f_2$:
\begin{equation*}
\begin{array}{l}
e_1(\lambda_{m-1}-\mu_{m-1}) +1 = e_2\mu_{m-1} \\
e_1(\lambda_m-\mu_m) = e_2\mu_m +1
\end{array}
\end{equation*}
Solving for $e_1$:
\begin{equation*}
\begin{array}{l}
\mu_m e_1(\lambda_{m-1}-\mu_{m-1}) +\mu_m = \mu_{m-1}e_1(\lambda_m-\mu_m) -\mu_{m-1} \\
\mu_{m}+\mu_{m-1} = e_1(\lambda_m \mu_{m-1}-\lambda_{m-1}\mu_m) = e_1
\end{array}
\end{equation*}
Thus, $e_1=\mu_{m-1}+\mu_m=q_1$ by Proposition \ref{pq}. Now, solving for $e_2$:
\begin{equation*}
\begin{array}{rl}
e_2 & \displaystyle =\frac{q_1(\lambda_m-\mu_m)-1}{\mu_m} \\
& \displaystyle = \frac{(\mu_{m-1}+\mu_m)(\lambda_m-\mu_m)-1}{\mu_m} \\
& \displaystyle = \frac{\lambda_m\mu_{m-1}+\lambda_m\mu_m-\mu_{m-1}\mu_m-\mu_m^2-(\lambda_m\mu_{m-1}-\lambda_{m-1}\mu_m)}{\mu_m} \\
& \displaystyle = \frac{\lambda_m\mu_m-\mu_{m-1}\mu_m-\mu_m^2+\lambda_{m-1}\mu_m}{\mu_m} \\
& \displaystyle = \lambda_{m-1}+\lambda_m - (\mu_{m-1}+\mu_m) \\
& \displaystyle = p_1 - q_1
\end{array}
\end{equation*}
The optimal way of obtaining strict transform $\tilde Y_1-\tilde X_1$ involves $\bar Q_2$ having strict transform $Y_1^{q_1}-X_1^{p_1-q_1}$ at $\Sigma(1,0,1)$. The $x$-blowup from level 0 to level 1 does not affect the ``$y$-parameter,'' so we would need a $y^{q_1}$ term in the definition of $\bar Q_2$. In order to get a jump value, we would need another term with the same value as $y^{q_1}$, hence we need the $x^{p_1}$ term. This shows $\bar Q_2=Q_2$ and the base step is done.

Assume the minimal valued generating sequence elements $\bar Q_j=Q_j$ for $j\leq i$, such that the strict transform $\tilde Y_{j-1}-c_j\tilde X_{j-1}$ is attained. Analogous to the base step, we wish to see what strict transform $Y_i^{e_1}-X_i^{e_2}$ at $\Sigma(i,0,1)$ has the total transform $\tilde X_i^{f_1}\tilde Y_i^{f_2}\left(\tilde Y_i - \tilde X_i\right)$ at $\Sigma(i+1,0,0)$. Essentially the same arguments as in the base case yields the necessity of having $Y_i^{q_i}-X_i^{p_i-q_i}$ as the strict transform of a minimal $\bar Q_{i+1}$ at $\Sigma(i,0,1)$. The inductive hypothesis plus Lemmas \ref{leveli} and \ref{leveli+} imply that transforming $Q_i$ is the minimal way to get a $Y_i$ at $\Sigma(i,0,1)$ while the other $\{Q_j\}$ can only yield $X_i$, for $j<i$. Hence $Q_i^{q_i}$ is a term in $\bar Q_{i+1}$. The remaining terms of equal value to $Q_i^{q_i}$ in Equation (\ref{genrecur}) in Section \ref{sectgs} are needed to have a jump in value. This shows $\bar Q_{i+1}=Q_{i+1}$.
\end{proof}

\begin{rem}
In addition to showing the $\{Q_i\}$ are minimal valued elements of $R$ with the appropriate strict transforms to introduce jump values, the proof of Lemma \ref{minimalvalue} also gives some justification as to why $Q_i$ was defined as it was in Equation (\ref{genrecur}) in Section \ref{sectgs} in the first place.
\end{rem}

\begin{lem} \label{genvalue}
The value of $Q_i$ is:
\begin{equation} \label{usualgenvalue}
\nu(Q_i)=q_{i-1}\nu(Q_{i-1})+\frac{1}{q_1 \cdots q_{i-1}}(\beta_i'-1)\beta_0
\end{equation}
where $2\leq i < g'$. This holds for all $i\geq 2$ in the Type 1 case ($g'=\infty$). This also holds for $i=g'$ in the Type 4.1 case.

For Types 2, 3 and 4.2,
\begin{equation*}
\nu(Q_{g'})=q_{g'-1} \nu(Q_{g'-1})+\nu(Y_{g'})
\end{equation*}
where the jump values are:
\begin{equation*}
\begin{array}{ll}
\text{Type 2:} & \displaystyle \nu(Y_g)= \frac{1}{q_1\cdots q_{g-1}}\left(\tilde \beta_g -1\right) \text{ , where } \tilde \beta_{g}\in\mathbb{R}\setminus\mathbb{Q} \\
\text{Type 3:} & \displaystyle \nu(Y_g)= \left(\lambda_{m-1}^{(g)}-\mu_{m-1}^{(g)},\lambda_{m-2}^{(g)}-\mu_{m-2}^{(g)}\right) \text{ , where } m=m_g \\
\text{Type 4.2:} & \displaystyle \nu(Y_{g+1})=\left(1,\frac{n}{q_1 \cdots q_g}\right) \text{ , where } n\in\mathbb{Z}
\end{array}
\end{equation*}
Also, in the Type 3 case: $\beta_0=q_1\cdots q_{g-1}(\mu_{m-1},\mu_{m-2})$, where $m=m_g$.
\end{lem}
\begin{proof}
First, $2\leq i < g'$. By Lemma \ref{strict},
\begin{equation*}
Q_i=\tilde X_{i-1}^{e_1}\tilde Y_{i-1}^{e_2}\left(\tilde Y_{i-1}-c_i\tilde X_{i-1}\right)u=X_i^{e_1+e_2+1}Y_iu
\end{equation*}
where $e_1,e_2\in\mathbb{N}$ and we ignore changes to the unit $u$.

The proof of Lemma \ref{beta'-1} shows:
\begin{equation*}
\nu(\tilde Y_{i-1}-c_i\tilde X_{i-1})=\frac{1}{q_1 \cdots q_{i-1}}\beta_i'\beta_0
\end{equation*}
The proof of Lemma \ref{minimalvalue} shows:
\begin{equation*}
Q_{i-1}^{q_{i-1}}=\tilde X_{i-1}^{e_1} \tilde Y_{i-1}^{e_2+1}u'
\end{equation*}
Lemma \ref{beta'-1} and the fact that $\nu(\tilde X_{i-1})=\nu(\tilde Y_{i-1})$ imply:
\begin{equation*}
q_{i-1}\nu(Q_{i-1})=\frac{e_1+e_2+1}{q_1\cdots q_{i-1}}\beta_0
\end{equation*}
Formula (\ref{usualgenvalue}) now follows by comparing the values of $Q_i$, $Q_{i-1}^{q_{i-1}}$ and $X_i^{e_1+e_2+1}Y_i$, noting that $\nu(Y_i)=\nu(\tilde Y_{i-1}-c_i\tilde X_{i-1})-\nu(X_i)$.

Now for the highest index $g'$. The arguments above hold for all $i$ in the Type 1 case, which justifies setting $g'=\infty$. For the remaining cases, the arguments above imply that $\nu(Q_{g'})=q_{g'-1}\nu(Q_{g'-1})+\nu(Y_{g'})$.

In the Type 2 case, $g'=g$ represents the introduction of an irrational value, encoded in $\displaystyle \nu\left(\tilde Y_{g-1}-c_g\tilde X_{g-1}\right)$. Let $\tilde \beta_g=\left[a_1^{(g)}, a_2^{(g)}, a_3^{(g)}, \ldots\right]$ be the irrational number that governs $G_g$, i.e. the $\beta$ used in Lemma \ref{XY} where $(\tilde X,\tilde Y)$ occurs in the limit and $(X,Y)=(\tilde X_{g-1},\tilde Y_{g-1}-c_g\tilde X_{g-1})$. Notice $\nu(X)$ and $\nu(Y)$ are related by $\tilde \beta_g$, so $\nu\left(\tilde Y_{g-1} - c_g \tilde X_{g-1}\right)=\tilde \beta_g \nu(\tilde X_{g-1})$. Then,
\begin{equation*}
\nu(Y_g)=\nu\left(\tilde Y_{g-1}-c_g\tilde X_{g-1}\right)-\nu(X_g)=\frac{1}{q_1\cdots q_{g-1}}\left(\tilde \beta_g-1\right)
\end{equation*}
which accounts for one less $z$-blowup to go from $\Sigma(g,0,0)$ to $\Sigma(g,0,1)$.

For Type 3 valuations, $g'=g$ and there are two cases to consider depending on whether $m_g$ is even or odd. Let $m=m_g$ be even. The odd case is similar. Normalize the valuation so that:
\begin{equation*}
\begin{array}{l}
\nu(x_{\Sigma(g,m-2,a_{m-1})})=(1,0) \\
\nu(y_{\Sigma(g,m-2,a_{m-1})})=(0,1)
\end{array}
\end{equation*}
Now use the matrix formulas (\ref{matrixvalues}) in Section \ref{sectdg} for $\beta=[a_1^{(g)},\ldots, a_{m-1}^{(g)}]$, omitting $a_m^{(g)}=\infty$. Notice $m-1$ is odd. Then,
\begin{equation*}
\left[
\begin{array}{c}
\nu(\tilde X_{g-1}) \\
\nu(\tilde Y_{g-1}-c_g\tilde X_{g-1})
\end{array}
\right]
=
\left[
\begin{array}{ll}
\mu_{m-1} & \mu_{m-2} \\
\lambda_{m-1} & \lambda_{m-2}
\end{array}
\right]
\left[
\begin{array}{c}
(1,0) \\
(0,1)
\end{array}
\right]
\end{equation*}
\begin{equation*}
\nu(X_g)=\nu(\tilde X_{g-1}) = (\mu_{m-1},\mu_{m-2})
\end{equation*}
\begin{equation*}
\nu(Y_g)=\nu(\tilde Y_{g-1}-c_g\tilde X_{g-1})-\nu(\tilde X_{g-1}) = (\lambda_{m-1}-\mu_{m-1},\lambda_{m-2}-\mu_{m-2})
\end{equation*}
If $m=m_g$ is odd, then $m-1$ is even and we normalize:
\begin{equation*}
\begin{array}{l}
\nu(x_{\Sigma(g,m-2,a_{m-1})})=(0,1) \\
\nu(y_{\Sigma(g,m-2,a_{m-1})})=(1,0)
\end{array}
\end{equation*}
\begin{equation*}
\left[
\begin{array}{c}
\nu(\tilde X_{g-1}) \\
\nu(\tilde Y_{g-1}-c_g\tilde X_{g-1})
\end{array}
\right]
=
\left[
\begin{array}{ll}
\mu_{m-2} & \mu_{m-1} \\
\lambda_{m-2} & \lambda_{m-1}
\end{array}
\right]
\left[
\begin{array}{c}
(0,1) \\
(1,0)
\end{array}
\right]
\end{equation*}
We see:
\begin{equation*}
  \nu(X_g)=(\mu_{m-1},\mu_{m-2})
\end{equation*}
and
\begin{equation*}
\nu(Y_g)=(\lambda_{m-1}-\mu_{m-1},\lambda_{m-2}-\mu_{m-2})
\end{equation*}
in this case as well.

To get $\beta_0$ in the Type 3 case, notice $\displaystyle \nu(X_g)=\frac{1}{q_1\cdots q_{g-1}}\beta_0$ so,
\begin{equation*}
\beta_0=q_1\cdots q_{g-1}\left(\mu_{m-1},\mu_{m-2}\right)
\end{equation*}

In the Type 4.1 case, it is possible to include $\{Q_i\}_{i=g+1}^{\infty}$ in a minimal generating sequence depending on how generating sequences are defined (see Section \ref{sectgs}). However, for our purposes we only need to consider up to level $g$. The $\displaystyle \{\nu(Q_i)\}_{i=g+1}^{\infty}$ don't add new denominators and they don't encode rank or rational rank jumps. From the value semigroup's perspective, they won't contribute anything. This justifies setting $g'=g$. Note that (\ref{usualgenvalue}) holds for $i=g'$ here by the previous arguments used for $i<g'$.

In the Type 4.2 case, (\ref{usualgenvalue}) doesn't work for $\nu(Q_{g+1})$ since there is a rank jump encoded in $\nu\left(\tilde Y_g - c_{g+1}\tilde X_g\right)$, hence also in $\nu(Y_{g+1})$ after the $z$-blowup. As discussed in Section \ref{sectgs}, the value group will be $\displaystyle \mathbb{Z}\times\frac{1}{q_1\cdots q_g}\mathbb{Z}$. The valuation is normalized such that $\nu(Y_{g+1})$ encodes the $\displaystyle \left(1,\frac{n}{q_1 \cdots q_g}\right)$ value, where $n\in\mathbb{Z}$ since $(1,*)>(0,0)$. Obviously, $g'=g+1$.
\end{proof}

\begin{rem}
If we set $q_0=1$, then Formula (\ref{usualgenvalue}) also works for $\displaystyle \nu(Q_1)=\beta_1=\beta_1'\beta_0=\frac{p_1}{q_1}\beta_0$. Compare this formula with similar formulas found in \cite{spiv}, \cite{ghk} and \cite{elh}. Setting $\nu(Q_i)=\beta_i$, the recursive Formula (\ref{betarecur}) from Section \ref{sectdg} is justified:
\begin{equation*}
\beta_i=q_{i-1}\beta_{i-1}+\frac{1}{q_1 \cdots q_{i-1}}(\beta_i'-1)\beta_0
\end{equation*}
\end{rem}

We will soon require use of the Frobenius problem from number theory, so this is a good place to introduce some definitions and results related to the Frobenius problem. We refer the reader to \cite{alf} and \cite{br} for more details.

The Frobenius problem asks the following question:

Given a set of distinct positive integers $\displaystyle \{a_1, \ldots, a_n\}$ with greatest common factor 1. What is the greatest integer that {\it cannot} be written as a linear combination of the $\{a_i\}$ over the non-negative integers? This greatest integer is called the {\it Frobenius number} and will be denoted $F(a_1, \ldots,a_n)$.

\begin{defin}
An integer $m$ is said to be {\it representable} by $\{a_1, \ldots, a_n\}$ if
\begin{equation*}
m = x_1a_1 + \cdots + x_na_n
\end{equation*}
for some $n$-tuple $\displaystyle (x_1, \ldots, x_n)\in({\mathbb N_0})^{n}$. Otherwise, $m$ is said to be {\it unrepresentable} by $\displaystyle \{a_1, \ldots, a_n\}$.
\end{defin}

An alternative way of stating the Frobenius problem is to say: find the largest unrepresentable integer given the set $\displaystyle \{a_1, \ldots, a_n\}$ defined above.

The solution is well-known for $n=2$, in which case the Frobenius number is just: $a_1a_2 - a_1 - a_2$. The Frobenius problem remains open for $n\geq3$. There are known lower and upper bounds for the Frobenius number. For our purposes, we are interested in upper bounds.

\begin{thm} \label{frobbound} (Brauer)
\\
Let $d_i:=\text{gcd}(a_1,\ldots,a_i)$. Let
\begin{equation*}
T(a_1,\ldots,a_n):=\sum_{i=1}^{n-1} a_{i+1}d_i / d_{i+1}
\end{equation*}
Then,
\begin{equation*}
F(a_1,\ldots,a_n) \leq T(a_1,\ldots,a_n) - \sum_{i=1}^{n} a_i
\end{equation*}
\end{thm}

\begin{proof}
See Theorem 3.1.2 of \cite{alf}.
\end{proof}

Note that if we can establish an upper bound, then the representable integers greater than the upper bound will be spaced evenly apart by 1 unit length. We wish to do something similar later on in the proof of Lemma \ref{lowerbeta}. There the Frobenius problem will be applied to a set of fractions related to the values of elements of a generating sequence. First the fractions are written in terms of their least common denominator, say $d$, then the Frobenius problem is applied to the numerators of a set of fractions with denominator $d$. Note that beyond the upper bound, all the representable fractions with the common denominator $d$ will be spaced evenly apart by $1/d$ units.

\begin{lem} \label{lowerbeta} Let $\nu$ be a non-divisorial valuation. For $1\leq i < g'$,
\begin{equation} \label{qibetai}
q_i \beta_i=\sum_{j=0}^{i-1} \alpha_j\beta_j
\end{equation}
for some $\alpha_j\in\mathbb{N}_0$. If $\nu$ is Type 4.1, then this holds for $1\leq i\leq g'$.
\end{lem}
\begin{proof}
Write $\displaystyle \beta_j=\frac{n_j}{q_1\cdots q_{i-1}}\beta_0$. The idea is to apply Theorem \ref{frobbound} to the set $\{n_0,\ldots,n_{i-1}\}$ and show that after writing $q_i\beta_i$ with denominator $\prod_{h=1}^{i-1} q_h$, the numerator of $q_i\beta_i$ is greater than the Frobenius number $F(n_0,\ldots,n_{i-1})$. The result immediately follows.

First, notice that $\text{gcd}(n_0\beta_0,\ldots,n_j\beta_0)=\text{gcd}(n_0,\ldots,n_j)\beta_0$. For clarity, we will essentially drop the common $\beta_0$ factor from all the values in the following arguments.

Using Lemma \ref{genvalue},
\begin{equation*}
q_i\beta_i=q_i\left(q_{i-1}\beta_{i-1}+\frac{p_i-q_i}{q_1\cdots q_i}\right)=\frac{q_iq_{i-1}n_{i-1}}{q_1\cdots q_{i-1}} + \frac{p_i-q_i}{q_1\cdots q_{i-1}}
\end{equation*}
so the desired numerator is $q_iq_{i-1}n_{i-1}+p_i-q_i$.

Write $\displaystyle \beta_j=\frac{\tau_j}{q_1\cdots q_j}\beta_0$. Notice $\displaystyle n_j=\tau_j\prod_{h=j+1}^{i-1} q_h$. Using Lemma \ref{genvalue}, it is easy to see that
\begin{equation} \label{taurecur}
\tau_j=q_jq_{j-1}\tau_{j-1}+p_j-q_j
\end{equation}
for $1\leq j \leq i-1$. Multiplying by $\displaystyle \prod_{h=j+1}^{i-1} q_h$,
\begin{equation} \label{nrecur}
n_j=q_{j-1}n_{j-1}+(p_j-q_j)\prod_{h=j+1}^{i-1} q_h
\end{equation}

Now $\beta_1=\beta_1'=p_1/q_1$, where $\text{gcd}(p_1,q_1)=1$, and $\beta_0=q_1/q_1$. Note that for Type 4.2, we use $\beta_0=(0,q_1/q_1)$ and the following arguments are similarly adjusted. For Type 3, $\beta_0$ has two coordinates and the semigroup values generated by $\{Q_i\}_{i=0}^{g'-1}$ are rational multiples of $\beta_0$. Some small adjustments need to be made in the following arguments, but we are in effect saying use $\beta_0= q_1/q_1 \cdot \beta_0$ in the Type 3 case.

So with $i=2$, $\text{gcd}(q_1,p_1)=\text{gcd}(n_0,n_1)$, and we see $\text{gcd}(n_0,n_1)=\text{gcd}(\tau_0 q_1,\tau_1)=1$. This is the base step of an induction on $i$ to show that $\text{gcd}(n_0,\ldots,n_{i-1})=1$. At level $i-1$, working with denominators $\prod_{h=1}^{i-2}q_h$, assume that:
\begin{equation*}
\text{gcd}(n_0,\ldots,n_{i-2})=\text{gcd}\left(\tau_0\prod_{h=1}^{i-2}q_h, \ldots ,\tau_{i-4}\prod_{h=i-3}^{i-2}q_h,\tau_{i-3}q_{i-2},\tau_{i-2}\right)=1
\end{equation*}
Multiplying by $q_{i-1}$ yields:
\begin{equation} \label{gcdfront}
\text{gcd}\left(\tau_0\prod_{h=1}^{i-1}q_h, \ldots ,\tau_{i-3}\prod_{h=i-2}^{i-1}q_h,\tau_{i-2} q_{i-1}\right)=q_{i-1}
\end{equation}
And so at level $i$, with denominators $\prod_{h=1}^{i-1} q_h$:
\begin{equation*}
\begin{array}{rl}
\displaystyle \text{gcd}(n_0,\ldots,n_{i-1}) & =\displaystyle \text{gcd}\left(\tau_0 \prod_{h=1}^{i-1}q_h, \ldots ,\tau_{i-3}\prod_{h=i-2}^{i-1}q_h,\tau_{i-2} q_{i-1},\tau_{i-1}\right)
\\\\
& = \text{gcd}\left(q_{i-1},\tau_{i-1}\right) \text{ by (\ref{gcdfront})}
\\\\
& = \text{gcd}\left(q_{i-1},q_{i-1}q_{i-2}\tau_{i-2}+p_{i-1}-q_{i-1}\right) \text{ by (\ref{taurecur})}
\\\\
& = \text{gcd}(q_{i-1},p_{i-1})=1
\end{array}
\end{equation*}
The induction is complete and thus we may apply the Frobenius upper bound in Theorem \ref{frobbound}.

Now working at level $i$ and denominators $\prod_{h=1}^{i-1} q_h$, let $d_j=\text{gcd}(n_0,\ldots,n_j)$ for $0\leq j\leq i-1$, and let $\displaystyle T=\sum_{j=0}^{i-2} n_{j+1}d_j/d_{j+1}$. Note that $d_0=n_0=\prod_{h=1}^{i-1}q_h$. Using gcd calculations similar to those previously done, it is easy to see that $d_j=\prod_{h=j+1}^{i-1}q_h$, hence $d_j/d_{j+1}=q_{j+1}$.
\begin{equation*}
\begin{array}{l}
\displaystyle T-\sum_{j=0}^{i-1}n_j =\sum_{j=0}^{i-2} n_{j+1}q_{j+1}-\sum_{j=0}^{i-1}n_j
\\\\
 = \displaystyle \sum_{j=1}^{i-1}n_jq_j-n_0-n_1-\sum_{j=2}^{i-1} n_j
\\\\
 = \displaystyle \sum_{j=1}^{i-1}n_jq_j-n_0-n_1-\sum_{j=2}^{i-1} n_{j-1}q_{j-1} - \sum_{j=2}^{i-1} (p_j-q_j)\prod_{h=j+1}^{i-1} q_h  \indent \text{by}\left.\right.(\ref{nrecur})
\\\\
 = \displaystyle n_{i-1}q_{i-1}-n_0-n_1-\sum_{j=2}^{i-1} (p_j-q_j)\prod_{h=j+1}^{i-1} q_h
\end{array}
\end{equation*}
Finally, comparing $T-\sum_{j=0}^{i-1}n_j$ with $q_iq_{i-1}n_{i-1}+p_i-q_i$, i.e. the numerator of $q_i\beta_i$:
\begin{equation*}
n_{i-1}q_{i-1} < q_iq_{i-1}n_{i-1}
\end{equation*}
and
\begin{equation*}
-n_0-n_1-\sum_{j=2}^{i-1}(p_j-q_j)\prod_{h=j+1}^{i-1} q_h \left.\right.<\left.\right. 0 \left.\right.<\left.\right. p_i-q_i
\end{equation*}
and the induction is complete by noting:
\begin{equation*}
F(n_0,\ldots,n_{i-1}) \leq T-\sum_{j=0}^{i-1}n_j<q_iq_{i-1}n_{i-1}+p_i-q_i
\end{equation*}

Now consider what happens at $g'$ to get the upper bound on $i$ for which Equation (\ref{qibetai}) is valid. Note that the argument above works so long as $\beta_i=\nu(Q_i)$ satisfies (\ref{usualgenvalue}). This is true in the Type 1 case for all $i$ by construction ($g'=\infty$).

For Types 2 and 4.2, there is no linear dependence relation possible between $\beta_{g'}$ and $\displaystyle \{\beta_j\}_{j=0}^{g'-1}$ because of a jump in rational rank or rank, respectively. Hence we have strict inequality $i<g'$.

In the Type 3 case, the strict inequality $i<g'$ follows from the fact that $e\beta_g$ cannot be written as $\displaystyle \frac{f}{q_1\cdots q_{g-1}}\beta_0$, where $e,f\in\mathbb{N}$. This will now be proved. By Lemma \ref{genvalue}, $\nu(Q_g)=q_{g-1}\nu(Q_{g-1})+\nu(Y_g)$ so it suffices to show $e\nu(Y_g)$ cannot be written as $\displaystyle \frac{f}{q_1\cdots q_{g-1}}\beta_0$. By Lemma \ref{genvalue},
\begin{equation*}
\beta_0=q_1\cdots q_{g-1}(\mu_{m-1},\mu_{m-2})
\end{equation*}
Assume $\displaystyle e\nu(Y_g)=\frac{f}{q_1\cdots q_{g-1}}\beta_0$. This implies:
\begin{equation*}
e\left(\lambda_{m-1}-\mu_{m-1},\lambda_{m-2}-\mu_{m-2}\right)=f\left(\mu_{m-1},\mu_{m-2}\right)
\end{equation*}
where $m=m_g$ and $\lambda_j/\mu_j$ is the $j$-th convergent of $\left[a_1^{(g)},a_2^{(g)},\ldots,a_{m-1}^{(g)}\right]$. Equating componentwise:
\begin{equation*}
\begin{array}{l}
e\lambda_{m-1}-e\mu_{m-1}=f\mu_{m-1} \\
e\lambda_{m-2}-e\mu_{m-2}=f\mu_{m-2}
\end{array}
\end{equation*}
Hence,
\begin{equation*}
\lambda_{m-1}/\mu_{m-1}=(e+f)/e=\lambda_{m-2}/\mu_{m-2}
\end{equation*}
which is a contradiction since consecutive convergents of a continued fraction cannot be equal.

In the Type 4.1 case, note that $\displaystyle q_g \beta_g=\frac{n}{q_1\cdots q_{g-1}}\beta_0$ by construction, where $n\in\mathbb{N}$, hence the lemma holds for $i=g'$ as well using the earlier Frobenius upper bound argument.
\end{proof}

\begin{lem} \label{nonredundant}
Only one element of a minimal generating sequence is necessary to introduce each new denominator for value jumps. More precisely, if $Q_i$ is part of a minimal generating sequence with $\displaystyle \nu(Q_i)=\frac{n_i}{q_1\cdots q_i}\beta_0$, where $n_i\in\mathbb{N}$, then a potential $\bar Q_i\in R$ with $\displaystyle \nu(\bar Q_i)=\frac{n}{q_1\cdots q_i}\beta_0$ would be redundant to include in a minimal generating sequence containing $Q_i$, where $n_i<n\in\mathbb{N}$.
\end{lem}

\begin{proof}
Use induction on $i$. We have $\nu(Q_0)=\beta_0$ and $\nu(Q_1)=\frac{p_1}{q_1}\beta_0$. Assume $\bar Q_1$ is the next element in the minimal generating sequence after $Q_0$ and $Q_1$ where $\bar Q_1$ and $Q_1$ have values with the same denominator $q_1$. The proof of Lemma \ref{lowerbeta} shows that $\beta_0$ and $\beta_1$ are sufficient to generate all values in $S$ with denominator $q_1$ and greater than or equal to $q_1\beta_1$. Thus, $\beta_1 < \nu(\bar Q_1) < q_1\beta_1$.

In order to have a jump in value, we need two or more equal-valued terms with a common value, say, $\varepsilon$. Let
\begin{equation*}
\bar Q_1 = u_1 Q_1^{e_1} + u_0 Q_0^{e_0} + \sum_j v_j T_j
\end{equation*}
where $e_0,e_1\in\mathbb{N}_0$, $\{u_j\}$ and $\{v_j\}$ are in $k$, and $\{T_j\}$ are monomials $Q_0^{f_{0,j}}Q_1^{f_{1,j}}$ with $f_{0,j}\geq 1$ and $f_{1,j} \geq 1$. All terms have the common value $\varepsilon$.

Assume both $u_1$ and $u_0$ are zero. Then we can factor out $Q_0$ or $Q_1$ from the remaining $T_j$ terms, contradicting the minimality of $\bar Q_1$. Assume both $u_1$ and $u_0$ are non-zero. Notice that $\beta_1$ introduced a new denominator $q_1$ which needs to be cleared in order for $\varepsilon$ to be representable by $\beta_0$. Hence $e_1=mq_1$, where $m\geq 1$, and this implies $\nu(\bar Q_1) > q_1\beta_1$, a contradiction. Assume only one of $u_1$ and $u_0$ is non-zero. For concreteness, let $u_1\neq0$ and $u_0=0$; the other way is similar. There exists a $v_j\neq 0$, since we need at least two terms with the same value $\varepsilon$ for a value jump. This contradicts the minimality of $\bar Q_1$ since we can factor out $Q_1$ from all the terms. Thus, $\bar Q_1$ does not exist.

Assume by inductive hypothesis that $\displaystyle \{Q_j\}_{j=0}^i$ form the beginnings of a minimal generating sequence. Let $\bar Q_i$ be the next minimal generating sequence element after $Q_i$ and whose value is assumed to not introduce a new denominator. That is, $\displaystyle \nu(\bar Q_i)=\frac{n}{q_1\cdots q_i}\beta_0$ for some $n\in\mathbb{N}$. The proof of Lemma \ref{lowerbeta} implies $\beta_i < \nu(\bar Q_i) < q_i\beta_i$.

In order to have a jump in value, we need two or more equal-valued terms with a common value, say, $\varepsilon$. Let
\begin{equation*}
\bar Q_i = \sum_{j=0}^{i} u_j Q_j^{e_j} + \sum_h v_h T_h
\end{equation*}
where $e_j\in\mathbb{N}_0$, $\{u_j\}$ and $\{v_h\}$ are in $k$, and $\{T_h\}$ are monomials in $\{Q_j\}_{j=0}^i$ consisting of at least two distinct factors.

Assume $u_i\neq 0$. Then we can either factor out $Q_i$ from all terms, contradicting the minimality of $\bar Q_i$, or at least one term does not have a factor of $Q_i$. In the latter case, $e_i=mq_i$, where $m\in\mathbb{N}$, in order for the terms to have a common value $\varepsilon$ since $\beta_i$ introduced a new factor $q_i$ in the denominator $\prod_{j=1}^{i}q_j$. This implies $\nu(\bar Q_i) > q_i\beta_i$, a contradiction. Hence $u_i=0$. For similar reasons, any monomial $T_h$ with non-zero $v_h$ (i.e. a {\it supported} $T_h$) cannot contain a factor of $Q_i$.

Assume $u_i=0$ and assume $Q_i$ does not show up in any supported $T_h$. By Lemma \ref{leveli+}, all the $\{Q_j\}_{j=0}^{i-1}$ will transform to the following form at $\Sigma(i,0,1)$: $uX_i^e$, where $u$ is a unit and $e\in\mathbb{N}$. Hence all the terms of $\bar Q_i$ will transform to the same form: $vX_i^f$, where $f\in\mathbb{N}$ is the same for all terms since they have common value $\varepsilon$, and where $v$ is a unit. Factoring out $X_i^f$ and absorbing all the units from each term into one unit, we see by Lemma \ref{beta'-1} that $\displaystyle \nu(\bar Q_i)=\frac{f}{q_1\cdots q_{i-1}}\beta_0$, a contradiction. Thus, $\bar Q_i$ does not exist.
\end{proof}

\begin{lem} \label{nonredundant234}
For Type 2 valuations, $Q_g$ is the last minimal generating sequence element. For Type 3 valuations, $Q_g$ is the last minimal generating sequence element. For Type 4.2 valuations, $Q_{g+1}$ is the last minimal generating sequence element.
\end{lem}
\begin{proof}
Let $\nu$ be Type 2. In order to have a jump value, we need two or more terms with the same value, say, $\varepsilon$. Assume $\bar Q_g$ is another minimal generating sequence element after $Q_g$. We have:
 \begin{equation*}
\bar Q_g=\sum_h u_hT_h
\end{equation*}
where $\{T_h\}$ are monomials in $\{Q_j\}_{j=0}^{g}$, and $\{u_h\}\in k$.

Assume $u_h\neq 0$ for at least one $T_h$ which contains a factor of $Q_g$. If all supported terms contain a factor of $Q_g$, then the minimality of $\bar Q_g$ is contradicted. Hence, there exists some supported term that does not contain a $Q_g$ factor. It is easy to see that a common value $\varepsilon$ is impossible here since $\nu(Q_g)$ is not a rational multiple of $\nu(Q_j)$ for $0\leq j \leq g-1$.

Assume $u_h=0$ for all the $\{T_h\}$ which contain a factor of $Q_g$. We only work with $\{T_h\}$ that are monomials in $\{Q_j\}_{j=0}^{g-1}$. Adapting the last part of the proof of Lemma \ref{nonredundant}, we see that $\displaystyle \nu(\bar Q_g) = \frac{f}{q_1\cdots q_{g-1}}$ for some $f\in\mathbb{N}$ and we also have
\begin{equation*}
q_{g-1}\nu(Q_{g-1}) < \nu(Q_g) < \nu(\bar Q_g)
\end{equation*}
This is a contradiction since values of the form $\displaystyle \frac{f}{q_1\cdots q_{g-1}}$ are representable by $\{Q_j\}_{j=0}^{g-1}$ as a consequence of the Frobenius upper bound argument used in the proof of Lemma \ref{lowerbeta}. Thus, $\bar Q_g$ does not exist.

Only slight changes are needed to make the proof for the Type 2 case suitable for the Types 3 and 4.2 cases. For Type 4.2, note that $\nu(Q_{g+1})$ cannot be written as a rational multiple of $\nu(Q_j)$ for $0\leq j \leq g$ because there is a rank jump encoded in $\nu(Q_{g+1})$. For Type 3, the proof of Lemma \ref{lowerbeta} implies $\nu(Q_g)$ cannot be written as $\displaystyle \frac{n}{q_1\cdots q_{g-1}}\beta_0$, which in turn implies that $\nu(Q_g)$ cannot be written in terms of $\nu(Q_j)$ for $0\leq j \leq g-1$. The remaining steps in the proof are completely analogous to what was done in the Type 2 case.
\end{proof}

\begin{thm} \label{minimalgs}
Given a non-divisorial valuation $\nu$, the $\displaystyle \{Q_i\}_{i=0}^{g'}$ form a minimal generating sequence from the perspective of generating the value semigroup $S$.
\end{thm}
\begin{proof}
See the discussion in Section \ref{sectgs}. In the Type 4.1 case, $g'=g$ since we are trying to generate $S$ rather than the value ideals $\{I_s\}$.
\end{proof}

\begin{thm} \label{unique} (Unique representation)
\\
Let $\nu$ be a non-divisorial valuation. Let $s\in S$. Assume $\nu$ is not Type 4.1. We may uniquely write:
\begin{equation*}
\begin{array}{rll}
\displaystyle s=\sum_{i=0}^{g'} \alpha_i \beta_i, & \text{ where } \alpha_0\in\mathbb{N}_0, \alpha_{g'}\in\mathbb{N}_0 \text{ and } 0 \leq \alpha_i \leq q_i-1 \text{ for } 1\leq i < g'
\end{array}
\end{equation*}
If $\nu$ is Type 4.1, then we may uniquely write:
\begin{equation*}
\begin{array}{rll}
\displaystyle s=\sum_{i=0}^{g} \alpha_i \beta_i, & \text{where}\left.\right. \alpha_0\in\mathbb{N}_0, \left.\right.\text{and}\left.\right. 0 \leq \alpha_i \leq q_i-1 \left.\right.\text{for}\left.\right. 1\leq i \leq g
\end{array}
\end{equation*}
\end{thm}

\begin{proof}
Generating sequences allow us to write $s=\sum \alpha_i \beta_i$ with $\alpha_i\in\mathbb{N}_0$. If $g'<\infty$, start from the penultimate index $g'-1$ down and repeatedly use Lemma \ref{lowerbeta} plus the division algorithm to establish the bounds on $\alpha_i$ for $i<g'$. That is, rewrite multiples of $q_i\beta_i$ in terms of $\displaystyle \{\beta_j\}_{j=0}^{i-1}$ then descend in $i$ and repeat the process at each lower step. If $g'=\infty$ (Type 1), any $s$ can be represented with finitely many $\{\beta_i\}$, hence the bounds on $\alpha_i$ can be established by the aforementioned process. Notice that $q_1\beta_1=p_1\beta_0$, so $\alpha_0$ can handle all the ``slack.''

Now we take care of the highest index $g'$, noting that Type 1 valuations have no highest index to worry about, so Type 1 is already done. By Lemma \ref{genvalue}, there is no linear dependence relation possible between $\beta_{g'}$ and $\{\beta_i\}_{i=0}^{g'-1}$ for valuations of Types 2 and 4.2. For Type 3 valuations, the proof of Lemma \ref{lowerbeta} shows that $e\beta_g$ cannot be represented by $\{\beta_i\}_{i=0}^{g-1}$, where $e\in\mathbb{N}$. Hence $\alpha_{g'}\in\mathbb{N}_0$ in these three cases.

For Type 4.1 valuations, notice $q_g\beta_g$ can be written in terms of $\displaystyle \{\beta_i\}_{i=0}^{g-1}$ using Lemma \ref{lowerbeta}, hence we get the bounds on $\alpha_{g'}$ by the earlier division algorithm argument.
\end{proof}

\begin{rem}
An alternative proof of this theorem for the non-discrete Type 1 case is given in \cite{ghk}.
\end{rem}

\section{concluding remarks} \label{sectcr}

This paper was originally intended to be the first of two papers stemming from the author's dissertation \cite{liphd} on the dual graphs and Poincar\'e series of valuations on function fields of dimension two. While this paper can stand on its own, this paper also logically sets up the second paper \cite{lisch} which uses dual graphs and generating sequences to analyze the Poincar\'e series of non-divisorial valuations.

It is natural to wonder why this paper focuses on the non-divisorial cases. The answer is two-fold. First, Spivakovsky's proofs gave greater details on the treatment of divisorial valuations, which left some room for exposition on the non-divisorial valuations. It is hoped that this paper will aid the readers who are interested in the dual graphs of valuations. Second, one of the goals of the author's dissertation was to classify valuations via their Poincar\'e series and the non-divisorial cases were the ones that needed attention for that purpose. However, much of the arguments in this paper can be easily adapted to the divisorial case.

Finally, it bears repeating that generating sequences are thought of as generating the value semigroup in this paper. See Section \ref{sectgs}. Readers should keep this in mind if they wish to work with generating sequences. In particular, this change in the definitions makes a significant difference in the case of Type 4.1 valuations.

\end{document}